\documentclass{amsart}

\usepackage{amsmath, amssymb, hyperref}
\usepackage[all]{xy} 
\usepackage{mathrsfs} 

\newcommand{\A}{\mathbb{A}}
\newcommand{\Af}{\mathbb{A}_f}
\newcommand{\Afp}{\mathbb{A}_f^{(\p)}}
\newcommand{\C}{\mathbb{C}}
\newcommand{\Q}{\mathbb{Q}}
\newcommand{\R}{\mathbb{R}}
\newcommand{\Z}{\mathbb{Z}}
\newcommand{\HH}{\mathcal{H}}
\newcommand{\OO}{\mathcal{O}}

\newcommand{\bbone}{\mathbf{1}} 

\newcommand{\g}{\mathfrak{g}}
\newcommand{\m}{\mathfrak{m}}
\newcommand{\n}{\mathfrak{n}}
\newcommand{\p}{\mathfrak{p}}
\newcommand{\pb}{\overline{\mathfrak{p}}}
\newcommand{\q}{\mathfrak{q}}

\newcommand{\ft}{\mathfrak{t}}

\newcommand{\into}{\hookrightarrow}

\DeclareMathOperator{\ord}{ord}
\DeclareMathOperator{\Res}{Res}
\DeclareMathOperator{\Max}{Max}
\DeclareMathOperator{\End}{End}
\DeclareMathOperator{\Ind}{Ind}
\DeclareMathOperator{\Hom}{Hom}
\DeclareMathOperator{\Ker}{Ker}
\DeclareMathOperator{\GL}{GL}
\DeclareMathOperator{\Lie}{Lie}
\DeclareMathOperator{\Ad}{Ad}
\DeclareMathOperator{\Rep}{Rep} 
\DeclareMathOperator{\Ver}{Ver} 
\DeclareMathOperator{\rank}{rank}

\newcommand{\an}{\mathrm{an}}
\newcommand{\la}{\mathrm{la}}

\newtheorem{theorem}    {Theorem}[subsection]
\newtheorem{corollary}  [theorem]{Corollary}
\newtheorem{lemma}      [theorem]{Lemma}
\newtheorem{proposition}[theorem]{Proposition}

\theoremstyle{definition}

\newtheorem{definition}[theorem]{Definition}
\newtheorem{assumption}[theorem]{Assumption}

\newtheorem{remark}    [theorem]{Remark}

\title{Overconvergent algebraic automorphic forms}

\author{David Loeffler}


\thanks{The author is grateful for the support of EPSRC postdoctoral fellowship EP/F04304X/1.}

\address{
DPMMS\\
Centre for Mathematical Sciences\\
University of Cambridge\\
Cambridge CB3 0WB, UK}

\email{d.loeffler.01@cantab.net}

\curraddr{
   Warwick Mathematics Institute\\
   Zeeman Building\\
   University of Warwick\\
   Coventry CV4 7AL, UK   }

\begin{document}

\begin{abstract}
I present a general theory of overconvergent $p$-adic automorphic forms and eigenvarieties for connected reductive algebraic groups $G$ whose real points are compact modulo centre, extending earlier constructions due to Buzzard, Chenevier and Yamagami for forms of $\GL_n$. This leads to some new phenomena, including the appearance of intermediate spaces of ``semi-classical'' automorphic forms; this gives a hierarchy of interpolation spaces (eigenvarieties) interpolating classical automorphic forms satisfying different finite slope conditions (corresponding to a choice of parabolic subgroup of $G$ at $p$). The construction of these spaces relies on methods of locally analytic representation theory, combined with the theory of compact operators on Banach modules.
\end{abstract}

\maketitle

 \begin{center}
 \line(1,0){250}
 \end{center}
 \emph{Note added December 2016}: The published version of this paper (Proc.\ London Math.\ Soc.\ \textbf{102} (2011), no. 2, p. 193--228) contained an error in section 2.6, which was subsequently spotted by Olivier Ta\"ibi. I have added a brief account of the error, and how to correct it, has been added to the end of this document. This will shortly be appearing as a corrigendum article in Proc.\ London Math.\ Soc.
  \begin{center}
  \line(1,0){250}
  \end{center}
  
\section{Introduction}

\subsection{Main results}

The aim of this paper is to obtain analogues of the results of \cite{Cpadic} and \cite{CMeigen} concerning the existence of $p$-adic analytic families of modular forms, where modular forms are replaced by automorphic forms on certain higher rank reductive groups. Given a reductive group $G$ defined over a number field $F$, with $G(F_v)$ compact modulo centre for all infinite places $v$; a finite prime $\p$ of $F$; and a parabolic subgroup $P \subseteq G(F_\p)$, we construct a family of $p$-adic Banach spaces of overconvergent automorphic forms, varying analytically over a certain weight space (parametrising characters of the Levi factor of $P$). This space has a Hecke action, and certain Hecke operators at $\p$ give compact operators. Applying the ``eigenvariety machine'' of \cite{buzzard-eigen}, one obtains a $p$-adic rigid space (an eigenvariety), which interpolates the Hecke eigenvalues of classical eigenforms which satisfy a finite slope condition. This finite slope condition depends on $P$, and is equivalent to the local factor at $\p$ of the associated automorphic representation being a subquotient of a representation parabolically induced from $P$; so taking $P$ to be larger gives an eigenvariety which may have smaller dimension, but sees more classical forms.

 We do not need any local assumptions on $G(F_\p)$; in particular, it need not be split or unramified. The special case $P = G$ recovers the space of classical forms, up to twisting by $p$-adic characters. On the other hand, when $G$ is a definite unitary group split at $p$, the space of overconvergent forms defined in \cite{chenevier-families} corresponds to the other extreme case where $P$ is a Borel subgroup. Thus the spaces corresponding to intermediate parabolic subgroups could perhaps be described as ``semi-classical'' overconvergent forms.

If $G$ satisfies the conditions of \cite{gross-algebraic}, so the arithmetic subgroups of $G$ are finite, then the eigenvariety we construct has dimension equal to that of the natural weight space, which is the space of characters of a compact open subgroup of the Levi factor of $P$; in particular, its dimension is the rank of the centre of the Levi factor. If $G$ possesses infinite arithmetic subgroups, then the eigenvariety has smaller dimension, and its image in weight space is contained in a closed subvariety of codimension $\ge 1$ whose geometry is closely related to Leopoldt's conjecture. 

The methods we shall use are an extension of ideas due to Buzzard \cite{buzzard-families} and Chenevier \cite{chenevier-families} for forms of $\GL_n$. The first key observation is that automorphic forms for the groups we consider can be defined as the space of functions
\[ f: G(F) \backslash G(\A) \to V,\]
where $V$ is an algebraic representation of $G$ over some coefficient field $E$ containing $F_\p$, which are locally constant on $G(F \otimes \R)$ and satisfy $f(g k) = k_\p^{-1} \circ f(g)$ for all $k$ in some compact open subgroup of $G(\Af)$. This space still makes sense if $V$ is an arbitrary representation of a compact open subgroup of $G(F_\p)$, rather than the whole group. Secondly, if we restrict to representations $V$ which are ``arithmetical'' (that is, their kernel contains an arithmetic subgroup), then the functor mapping $V$ to the space of automorphic forms with values in $V$ is exact and commutes with base change. So it is sufficient to construct a family of representations over an appropriate weight space, which is carried out in \S 2 (via a form of locally analytic parabolic induction). In \S 3 we recall the necessary global theory to define spaces of automorphic forms with values in these modules, and apply the eigenvariety machine to these spaces.

We also prove a control theorem, analogous to the main result of \cite{C-CO}, showing that overconvergent automorphic eigenforms whose weight is classical and whose Hecke eigenvalues at $\p$ have small valuation are in fact classical forms. This implies that points arising from classical automorphic forms are dense in the eigenvariety whenever locally algebraic weights have a certain accumulation property in weight space. We give an example to show that this is not always the case.

\subsection{Relation to other constructions}

As mentioned above, the methods of this paper are a generalisation of those used in several earlier works dealing with specific examples of compact-modulo-centre reductive groups; the main examples are \cite{buzzard-families}, \cite{buzzard-eigen}, \cite{chenevier-families} and \cite{yamagami}. This is considered in detail in section \ref{chap:examples}, where we show how to recover some of the results of these papers from our approach.

A very different approach to the problem of $p$-adic interpolation can be found in the works of Emerton (\cite{emerton-interpolation} and the papers referenced therein). This theory, which does not require $G(F \otimes \R)$ to be compact modulo centre but at present requires $G(F_\p)$ to be quasi-split, is in a sense opposite to the approach followed here: for each $i \ge 0$, Emerton builds a locally analytic representation $\widetilde H^i_{F_\p-\la}$ of the whole of $G(F_\p)$ which ``sees'' all automorphic representations of $G$ which are cohomological in degree $i$. The eigenvariety is then constructed by applying the Jacquet module functor of \cite{emerton-jacquet} to this representation. We study the relation between the two constructions in \S \ref{ssect:emerton}. For compact-modulo-centre groups the interesting cohomology is all in degree 0, and we show that there is a bijection between the finite slope eigenspaces of the space of overconvergent automorphic forms constructed in this paper and the fibres of Emerton's Jacquet module $J_{\overline{P}}(\widetilde H^0_{F_\p-{\la}})$. In future work I hope to show how the methods of \cite{emerton-interpolation} can be extended to the case of a general parabolic subgroup, which would allow the results of this paper to be extended to the case where $G(F \otimes \R)$ is not compact modulo centre.

A third approach can be found in recent preprints of Ash and Stevens \cite{ash-stevens07}. Again this does not impose any restrictions on $G$ at infinity, but assumes $F = \Q$ and $G(F_\p)$ split. The relation between their approach and the methods of the present paper is not entirely clear at present, but the ``universal highest weight modules'' that they define (using spaces of locally analytic distributions on Iwahori subgroups) are formally similar to the duals of the locally analytic induced representations we consider.

\section{Families of induced representations}
\label{sect:local-families}

Let $L$ be a finite extension of $\Q_p$, with ring of integers $\OO_L$ and uniformiser $\pi_L$.

We recall some basic concepts from $p$-adic function theory, following \cite[\S 2.1]{emerton-memoir}. We fix a field $E \supseteq L$, complete with respect to a discrete valuation extending that of $L$; this will be the coefficient field for all the representations we shall consider. Let $\mathcal{X}$ be an affinoid rigid space over $L$; then we define $C^\an(\mathcal{X}, E) = \OO(\mathcal{X})\ \widehat\otimes_L\ E$, the space of $E$-valued analytic functions on $\mathcal{X}$. For any $E$-Banach space $V$, we define $C^\an(\mathcal{X}, V) = C^\an(\mathcal{X}, E)\ \widehat\otimes_E\ V$, the space of $V$-valued analytic functions on $\mathcal{X}$ \cite[def.~2.1.9]{emerton-memoir}. 

We now let $X$ be a topological space, which we suppose to be strictly paracompact and Hausdorff. Following \cite[\S 2.1]{emerton-memoir}, we define an $L$-analytic partition of $X$ (of dimension $d$) to be the data of an indexing set $I$ and family of triples $(U_i, \mathcal{U}_i, \phi_i)_{i \in I}$, where: $\{U_i\}_{i \in I}$ is a cover of $X$ by disjoint open sets; for each $i$, $\mathcal{U}_i$ is a rigid-analytic affinoid ball of dimension $d$ over $L$; and $\phi_i$ is a homeomorphism $\mathcal{U}_i(L) \stackrel{\sim}{\to} U_i$. If $\underline{U} = (U_i, \mathcal{U}_i, \phi_i)_{i \in I}$ and $\underline{V} = (V_j, \mathcal{V}_j, \phi_j)_{j \in J}$ are both $L$-analytic partitions of $X$, we say $\underline{U}$ is a refinement of $\underline{V}$ if there is a function $\sigma: I \to J$ such that $U_i \subseteq V_{\sigma(i)}$ for all $i \in I$, and the inclusion $U_i \into V_{\sigma(i)}$ extends to an open embedding of rigid spaces $\mathcal{U}_i \into \mathcal{V}_{\sigma(i)}$ (which is necessarily unique if it exists, since $U_i$ is Zariski-dense in $\mathcal{U}_i$). We define two $L$-analytic partitions to be equivalent if they posess a common refinement. Then we can define a locally $L$-analytic manifold as a topological space $X$ endowed with an equivalence class of $L$-analytic partitions; we refer to the given equivalence class as the atlas of $X$, and its elements as charts. The $L$-points of any smooth algebraic variety or rigid space over $L$ are naturally a locally $L$-analytic manifold \cite[5.8.10]{bourbaki-VAR}.

If $X$ is a locally $L$-analytic manifold and $V$ is a finite-dimensional $E$-vector space, then for any chart $(U_i, \mathcal{U}_i, \phi_i)_{i \in I}$ of $X$, we can define the space $\prod_{i \in I} C^\an(\mathcal{U}_i, V)$, whose elements are naturally continuous functions on $X$. Refinements of charts naturally give rise to morphisms between these spaces; and since the atlas of $X$ is by definition a directed set under refinement, we can define the space $C^\la(X, V)$ to be the direct limit. This is the space of {\it locally $L$-analytic $V$-valued functions on $X$.}

Let $G$ be a connected reductive algebraic group over $L$. We refer to \cite[\S 5]{borel-tits} for the structure theory of reductive algebraic groups over a non-algebraically-closed field. Let $S$ be a maximal split torus in $G$; the adjoint action of $S$ on $\g = \Lie G$ gives a relative root system whose Weyl group can be identified with the quotient $N_G(S)/Z_G(S)$. Let us choose a system of simple roots $\Delta$; this defines a minimal parabolic subgroup $P_{\rm min}$ of $G$, containing $S$. 

Fix a parabolic subgroup $P \supseteq P_{\rm min}$, and let $N$ be its unipotent radical, $M = P/N$ the Levi quotient, $M^{ss}$ the semisimple part of $M$, $H = M / M^{ss}$ the maximal torus quotient of $M$. We shall call a positive (relative) root {\bf free} if $-\alpha$ is not a root of $P$; write  $\Delta_f$ for the set of free simple positive roots.\footnote{Note that our normalisations give a correspondence $P \leftrightarrow \Delta_f$ between parabolics and subsets of $\Delta$ which is bijective and inclusion-reversing; in particular every root is free for $P = P_{\rm min}$, and no root is free for $P=G$.} 

The choice of $P$ and $S$ determines an opposite parabolic $\overline{P}$, the conjugate of $P$ by the longest element in the Weyl group of $S$, such that $\overline{P} \cap  P$ is a lifting of $M$ to a subgroup of $G$ (containing the centraliser of $S$). We write $\overline{N}$ for the unipotent radical of $\overline{P}$. We set $T = Z(M) \cap S$, a maximal split torus in the centre of $M$.

For $\Gamma$ a locally $L$-analytic group, let $\Rep_{\rm la,c}(\Gamma)$ denote the category of locally $L$-analytic representations of $\Gamma$ on $E$-vector spaces of compact type, as defined in \cite{ST-distributions}. We let $\Rep_{\rm la,fd}(\Gamma)$ denote the full subcategory of such representations which are finite-dimensional over $E$.

\subsection{Parahoric induction}

We shall consider representations of compact open subgroups of $G(L)$ satisfying the following condition:

\begin{definition}[{\cite[\S 1.4]{casselman-book}}]\label{def:iwahori-decomp}
We say that a compact open subgroup $U \subseteq G(L)$ has an {\it Iwahori factorisation} (with respect to the parabolics $\overline{P}$ and $P$) if:
\begin{enumerate}
 \item The natural multiplication map $\overline{N}_U \times M_U \times N_U \to U$ is an isomorphism of locally $L$-analytic manifolds, where $\overline{N}_U = \overline{N} \cap U$, $M_U = M \cap U$, $N_U = N \cap U$;
 \item If $z \in T(L)$ is such that $|\alpha(z)| \le 1$ for all $\alpha \in \Delta_f$, then $z N_U z^{-1} \subseteq N_U$ and $z^{-1} \overline{N}_U z \subseteq \overline{N}_U$.
\end{enumerate}
\end{definition}

We shall see in the next subsection that such subgroups $U$ do indeed exist, for any connected reductive group $G$ over $L$.

\begin{definition}
Let $U$ be a subgroup of $G$ with an Iwahori factorisation $U = \overline{N}_U \times M_U \times N_U$, and let $V \in \Rep_{\rm la, fd}(M_U)$. We define $\Ind_{\overline{P}}^{U} V$ to be the space
\[ \left\{ f \in C^{\la}(U, V)\ \middle|\ f(\overline{n} m g) = m \circ f(g)\right\}\]
where the functional equation is to be satisfied for all $g \in U$, $m \in M_U$, and $\overline{n} \in \overline{N}_U$.
\end{definition}

By results of \cite[\S 4]{feaux}, this is in $\Rep_{\rm la,c}(U)$, and its underlying vector space can be identified with $C^{\la}(N_U, V)$. This identification can be viewed as arising from a right action of $U$ on $N_U$, obtained by composing right multiplication with the projection onto the $N_U$ factor in the Iwahori factorisation. 

Concretely, if $\Psi_M$ and $\Psi_N$ denote the projection maps from $U$ onto $M_U$ and $N_U$ respectively, the action of $\gamma \in U$ on $f \in C^{\la}(N_U, V)$ is given by 
\[ (\gamma \circ f)(n) = \Psi_M(n\gamma) \circ f(\Psi_N(n\gamma)).\]

Let $H$ be the torus $M / M^{ss}$ over $L$. The image of $M_U$ in $H(L)$ is a compact open subgroup $H_U$, and any locally $L$-analytic character $H_U \to E^\times$ defines a locally $L$-analytic character of $M_U$ by inflation. From the above formula for the $U$-action the following is immediate:

\begin{proposition} \label{prop:twists}
 If $\tau$ is a locally $L$-analytic character of $H_U$, then $\Ind_{\overline{P}}^{U}(V \otimes \tau)$ is isomorphic to $C^{\la}(N_U, V)$ with the twisted $U$-action $\circ_{\tau}$ given by
\[ (\gamma \circ_\tau f)(n) = \tau( \Psi_M(n\gamma)) \cdot (\gamma \circ f)(n).\]
\end{proposition}

\subsection{Analytic vectors}

We now choose a sequence of compact open subgroups of $G(L)$, following \cite[\S 4.1]{emerton-jacquet}.

\begin{definition}\label{def:decomposable}
Let $U$ be a compact open subgroup of $G(L)$. We will say $U$ is {\it decomposable} if:
\begin{enumerate}
 \item $U$ is a good $L$-analytic open subgroup of $G$, in the sense of \cite[\S 5.2]{emerton-memoir}, so there exists a Lie $\OO_L$-lattice $\g_U$ in $\g = \Lie G$ such that the exponential map converges on the affinoid ball associated to $\mathfrak{g}_U$ and identifies $\mathfrak{g}_U$ with $U$;
 \item The lattice $\g_U$ decomposes as a direct sum 
\[ \g_U = \m_U \oplus \bigoplus_{\alpha \in \pm \Phi_f} \mathfrak{n}_{U,\alpha}\]
where $\Phi_f$ is the set of positive free restricted roots, and $\mathfrak{m}_U$ and $\mathfrak{n}_{U, \alpha}$ denote Lie $\OO_L$-lattices in the Lie algebras of $M$ and the root subgroups $N_\alpha$.
\end{enumerate}
\end{definition}

\begin{proposition}
There exists a decomposable subgroup of $G(L)$.
\end{proposition}

\begin{proof} We follow the construction given in \cite[4.1.6]{emerton-jacquet}. Let $\mathfrak{u}$ be an arbitrary $\OO_L$-lattice in $\g$, and let $a \in L^\times$. Then we can define $\mathfrak{u}' = (a\mathfrak{u} \cap \mathfrak{m}) \oplus \bigoplus_{\alpha \in \pm \Phi_f} (a\mathfrak{u} \cap \mathfrak{n}_\alpha)$. If the valuation of $a$ is large enough, then this is a Lie sublattice and corresponds to a good $L$-analytic compact open subgroup of $G(L)$. By construction, this is decomposable.
\end{proof}

If $U$ is decomposable, then it certainly has an Iwahori factorisation: part (i) of definition \ref{def:iwahori-decomp} follows as the sublattices $\m_U$, $\mathfrak{n}_U = \bigoplus_{\alpha \in \Phi_f} \mathfrak{n}_{U,\alpha}$, and $\overline{\mathfrak{n}}_U = \bigoplus_{\alpha \in \Phi_f} \mathfrak{n}_{U,-\alpha}$ exponentiate to good $L$-analytic open subgroups $M_U$, $N_U$, $\overline{N}_U$ of $M$, $N$, $\overline{N}$ respectively, which clearly satisfy $\overline{N}_U \times M_U \times N_U = U$; and part (ii) follows since for $z \in S(L)$, $\Ad z$ acts on the root space $n_\alpha$ as multiplication by $\alpha(z)$, so if $|\alpha(z)| \le 1$ for all free $\alpha$ then $\Ad z$ preserves the lattices $\mathfrak{n}_{U, \alpha}$ for each positive root and hence preserves $N_U$ (and similarly for $\Ad z^{-1}$ on $\overline{N}_U$). 

\begin{definition}\label{def:admissible}
An {\it admissible subgroup} of $G$ is a compact open subgroup $G_0 \subseteq G(L)$ which admits an Iwahori factorisation, and such that there exists a normal subgroup $G_1 \trianglelefteq G_0$ which is decomposable.
\end{definition}

Any decomposable subgroup is certainly itself admissible, so admissible subgroups exist. We now fix a choice of an admissible subgroup $G_0$, and a decomposable normal subgroup $G_1$. If $\mathfrak{g}_1$ is the Lie sublattice corresponding to $G_1$, we let $G_i$ for $i \in \Z_{> 1}$ denote the good $L$-analytic open subgroup corresponding to the sublattice $\pi_L^{i-1} \mathfrak{g}_i$; this is clearly also a decomposable subgroup of $G(L)$ normal in $G_0$. Let $\overline{N}_i, M_i, N_i$ denote the factors in the Iwahori factorisation of $G_i$, and for $i \ge 1$ let $\mathcal{N}_i$ denote the rigid analytic subgroup underlying $N_i$.

Recall the definition of {\it analytic vectors} for a good $L$-analytic subgroup $U$ in a locally $L$-analytic representation $V$ of a locally $L$-analytic group $G$ \cite[\S 3.4]{emerton-memoir}; these are the vectors $v \in V$ for which the orbit map $\rho_v: G \to V$ restricts to a function in $C^\an(\mathcal{U}, V')$, where $\mathcal{U}$ is the rigid analytic group underlying $U$ and $V'$ is a Banach space with a continuous embedding $V' \into V$.

For all $k \ge 1$, let
\[\mathcal{D}_k = \bigcup_{[x] \in N_0 / N_k} x \mathcal{N}_k.\]
This is an open affinoid subspace of $N$, with $\mathcal{D}_k(L) = N_0$, and $G_0$ acts on $\mathcal{D}_k$ on the right via rigid-analytic automorphisms.

\begin{proposition} \label{prop:analytic-vectors}
If $V \in \Rep_{\rm la,fd}(M_0)$, then for all $k \ge 1$, the space of $G_k$-analytic vectors 
\[\left(\Ind_{\overline{P}}^{G_0} V\right)_{G_k-\an}\]
is exactly the image of $C^\an(\mathcal{D}_k, V_{M_k-\an})$ in $C^{\la}(N_0, V) \cong \Ind_{\overline{P}}^{G_0} V$.
\end{proposition}

\begin{proof}
As $G_k$ is normal in $G_0$ for each $k$, $\Ind_{\overline{P}}^{G_0} V$ decomposes as a $G_k$-representation into a direct sum of representations of the form $\Ind_{\overline{P}}^{G_k} V$, indexed by the cosets $N_0 / N_k$.

As in the proof of prop.~2.1.2 of \cite{emerton-jacquet2}, $(\Ind_{\overline{P}}^{G_k} V)_{G_k-\an}$ is the preimage of $\Ind_{\overline{P}}^{G_k} V$ under the natural inclusion $C^\an(\mathcal{G}_k, V) \into C^{\la}(G_k, V)$. It is clear that any element of $C^\an(\mathcal{G}_k, V)$ which maps into $\Ind_{\overline{P}}^{G_k} V$ must have image contained in $V_{M_k-\an}$. The converse follows from the Iwahori factorisation of $G_k$. 
\end{proof}

Let use the abbreviation $\Ind(V)_k$ for $\left(\Ind_{\overline{P}}^{G_0} V\right)_{G_k-\an}$. Since $\mathcal{D}_k$ is evidently reduced, the preceding proposition allows us to regard this as a Banach space in the supremum norm, once we fix a choice of norm on $V$. Let us fix a norm that is $M_0$-equivariant, and such that $V$ is orthonormalisable as an $E$-Banach space; such norms always exist, since $M_0$ is compact. The resulting norm on $\Ind(V)_k$ is then clearly $G_0$-invariant.

We conclude this subsection by writing down a specific basis for $\Ind(V)_k$. Suppose that $k$ is sufficiently large that $V_{M_k-\an} = V$; such a $k$ certainly exists, since $V$ is finite-dimensional. Let $x \in N_0 / N_k$. Then $U_x = C^\an(x \mathcal{N}_k, V)$ is a closed $G_k$-stable subspace of $\Ind(V)_k$, and $\Ind(V)_k$ is isometrically isomorphic to the orthogonal direct sum $\bigoplus_{x \in N_0 / N_k} U_x$. The spaces $U_x$ are clearly all isomorphic as Banach spaces, so we concentrate on the identity coset.

Let $\Phi_f$ denote the set of free positive roots. For each $\alpha \in \Phi_f$, let $d(\alpha)$ be the multiplicity of $\alpha$ (the dimension of the $\alpha$-root space $\mathfrak{n}_\alpha$ of $\g$) and let us choose a basis $x_{\alpha, 1}, \dots, x_{\alpha, d(\alpha)}$ for $\n_{k, \alpha} = \mathfrak{n}_k \cap \mathfrak{n}_\alpha$. Then $C^\an(\mathcal{N}_k, E)$ is equal to the Tate algebra $E \left\langle \{x_{\alpha, i}\}_{\alpha \in \Phi_f, 1 \le i \le d(\alpha)}\right\rangle$.

Let $J$ be the set of functions from the set of pairs $ \{(\alpha, i)\}_{\alpha \in \Phi_f, 1 \le i \le d(\alpha)}$ to $\Z_{\ge 0}$, and for $\underline{j} \in J$, write $\underline{X}^{\underline{j}}$ for the product $\prod_{\alpha, i} (x_{\alpha, i})^{\underline{j}(\alpha, i)}$. Let $v_1, \dots, v_r \in V$ be vectors which form an orthonormal basis for $V$ as an $E$-Banach space (with the $M_0$-invariant norm fixed above). 

\begin{proposition}\label{prop:basis}
The vectors
\[ \left\{ \underline{X}^{\underline{j}} \otimes v_{j'}\ \middle|\ \underline{j} \in J, 1 \le j' \le r\right\}\]
are an orthonormal basis for $C^\an(\mathcal{N}_k, V)$ as an $E$-Banach space.
\end{proposition}

\subsection{Families of twists}

We now come to the main construction of this section. Let $H_0$ be the image of $M_0$ in the torus $H$, and let $\widehat H_0$ be the rigid space parametrising locally $L$-analytic characters of $H_0$, as constructed in \cite[\S 6.4]{emerton-memoir}. By construction, the space $\widehat H_0$ is a rigid space over $L$, isomorphic over $\C_p$ to a finite \'etale cover of the open unit disc $\mathbf{B}^\circ(1)$. Its moduli space interpretation furnishes it with a ``universal character'' $\Delta: H_0 \to C^\an(\widehat H_0, E)^\times$, whose composition with the evaluation map at any point $z \in \widehat H_0(\C_p)$ is the character $H_0 \to \C_p^\times$ corresponding to $z$.

The images of the subgroups $M_k$ in $H$ give a basis $\{H_k\}_{k \ge 0}$ of neighbourhoods of the identity in $H$; and for $k \ge 1$, $H_k$ is a good $L$-analytic open subgroup of $H$. Since affinoids are quasi-compact, it is clear that for any affinoid $X \subseteq \widehat H_0$, there is some $k \ge 1$ such that the composite $\Delta_X: H_0 \to C^\an(X, E)^\times$ is analytic on cosets of $H_k$; let $k(X)$ be the least such integer.

\begin{definition}\label{def:main-construction}
 If $X \subseteq \widehat H_0$ is an affinoid defined over $E$, and $k \ge k(X)$, let 
\[ \mathcal{C}(X, V, k) = C^\an(\mathcal{D}_k \times X, V_{M_k-\an}).\]
We endow this with a $G_0$-action by defining, for $\gamma \in G_0$, 
\[ (\gamma \circ f)(n, x) = \Delta_X(\Psi_M(n\gamma))(x) \cdot \left[ \Psi_M(n\gamma) \circ f( \Psi_N(n\gamma), x)\right].\]
\end{definition}

\begin{proposition}
For every $\tau \in X(E)$ (giving a map $C^\an(X, E) \to E$) we have 
\[ \mathcal{C}(X, V, k)\ \widehat\otimes_{C^\an(X, E), \tau}\ E = \Ind(V \otimes \tau)_k.\]
as $E$-Banach spaces and as representations of $G_0$.
\end{proposition}

\begin{proof} We note that 
\[ C^\an(\mathcal{D}_k \times X, V_{M_k-\an}) = C^\an(\mathcal{D}_k, V_{M_k-\an})\ \widehat\otimes_E\ C^\an(X, E).\]
By proposition \ref{prop:analytic-vectors}, we therefore have 
\[ \mathcal{C}(X, V, k) = \Ind(V \otimes \tau)_k\ \widehat\otimes_E\ E\]
with the $G_k$-action on the second factor given by $\tau \circ \Psi_M$. The result now follows from proposition \ref{prop:twists}.
\end{proof}

We now establish some simple properties of the spaces $\mathcal{C}(X, V, k)$. 

\begin{proposition}\label{prop:weight-fiddle}
With the notation above, suppose $\tau$ is an $E$-point of $\widehat H_0$ and $k \ge k(\tau)$. Then the spaces $\mathcal{C}(X, V \otimes \tau, k)$ and $\mathcal{C}({X + \tau}, V, k)$ are isomorphic as Banach $C^\an(X, E)$-modules and as representations of $G_0$.
\end{proposition}

\begin{proof} This is immediate from the construction. \end{proof}

\begin{proposition} \label{prop:projective}
If $X$ is a reduced affinoid over $E$, the space $\mathcal{C}(X, V, k)$ admits a $G_0$-invariant norm in which it is orthonormalisable as a $C^\an(X, E)$-Banach module.
\end{proposition}

\begin{proof}
If $X$ is reduced, then so is $X \times \mathcal{D}_k$, and hence $C^\an(\mathcal{D}_k \times X, E)$ has a canonical supremum norm, in which it is orthonormalisable; and since $V$ is orthonormalisable, so is their completed tensor product $C^\an(\mathcal{D}_k \times X, V)$. By compactness of $H_0$, the character $\Delta_X: H_0 \to C^\an(X, E)^\times$ must have image contained in the norm 1 elements of $C^\an(X, E)$; so the norm is $G_0$-invariant.
\end{proof}

\subsection{Action of the torus}
\label{ssect:torus-action}

We shall now show that the Banach modules $\mathcal{C}(X, V, k)$ constructed above admit an action of a certain semigroup of non-integral elements of the split torus $T = Z(M) \cap S$, and that these act via compact operators. (This will allow us to give our overconvergent $\mathfrak{p}$-adic automorphic forms a Hecke action at $\mathfrak{p}$.)

\begin{proposition}\label{prop:centre-action}
Let $k \ge 1$ and $z \in T(L)$.
\begin{enumerate}
 \item If $|\alpha(z)| \le 1$ for all $\alpha \in \Delta_f$, then the embedding $z N_k z^{-1} \subseteq N_k$ extends to a map of rigid spaces $\mathcal{N}_k \to \mathcal{N}_k$.
 \item If $|\alpha(z)| < 1$ for all $\alpha \in \Delta_f$, then the map of (ii) factors through the embedding $\mathcal{N}_{k+1} \into \mathcal{N}_k$.
\end{enumerate}
\end{proposition}

\begin{proof} Recall that $\mathcal{N}_k$ is the affinoid ball associated to an $\OO_L$-lattice $\mathfrak{n}_k \subseteq \Lie N$, and this lattice is equal to the direct sum of the lattices $\mathfrak{n}_{k, \alpha}$ which are its intersections with the positive root spaces $N_\alpha$. On each such root space $\Ad z$ acts as multiplication by $\alpha(z) \in \OO_L$, so it certainly preserves the affinoid ball attached to any $\OO_L$-lattice. Moreover if $|\alpha(z)| < 1$, we must have $\alpha(z) \in \pi \OO_L$, so the map has image contained in the affinoid ball corresponding to $\pi \mathfrak{n}_{k, \alpha}= \mathfrak{n}_{k+1, \alpha}$.
\end{proof}

Let $T_0 = T(L) \cap M_0$. Since $T$ is a torus, $T(L)$ is topologically finitely generated and hence $T(L) / T_0$ is a finitely generated abelian group; thus we can choose a finite set of elements of $T(L)$ whose images in $T(L) /T_0$ are a basis for $(T(L) / T_0) \otimes \Q$. Let $\Sigma$ be the subgroup of $T(L)$ generated by these elements, which is a free abelian group of rank equal to the split rank of $Z(M)$. We can regard any $M_0$-representation as a representation of $\Sigma M_0$, on which $\Sigma$ acts trivially.

\begin{remark} If (as will often be the case) $V$ is the restriction to $M_0$ of a representation of $M(L)$, then this is clearly not the ``natural'' action of $\Sigma$; but removing this requirement would make the arguments of the next chapter much messier.
\end{remark}

\begin{definition}
We define
\[\begin{aligned}
 \Sigma^{+} &= \{z \in \Sigma : |\alpha(z)| \ge 1\ \forall \alpha \in \Delta_f\}, & \Sigma^{++} &= \{z \in \Sigma : |\alpha(z)| > 1\ \forall \alpha \in \Delta_f\}.
\end{aligned}\]
\end{definition}

\begin{proposition}
The semigroup $\Sigma^{++}$ is not empty, and contains a basis for $\Sigma \otimes \Q$.
\end{proposition}

\begin{proof}
It follows from the discussion of \cite[\S 1.4]{emerton-memoir} that the maps $\ord \circ \alpha$ for $\alpha \in \Delta_f$ are linearly independent functionals on $X \otimes \Q$, where $X = T(L) / T_0$. Hence the set of points satisfying the inequalities $\ord_p \alpha < 0$ for all $\alpha \in \Delta_f$ is a nonempty open cone in $X \otimes \R$. Since $\Sigma$ has finite index in $X$, its image in $X \otimes \R$ certainly has nonempty intersection with this cone, and the intersection contains a basis for $\Sigma \otimes \Q$.
\end{proof}

If $z \in \Sigma^+$, then by part (ii) of definition \ref{def:iwahori-decomp} applied to $z^{-1}$, we have $z^{-1} N_k z \subseteq N_{k}$ and $z \overline{N}_k z^{-1} \subseteq \overline{N}_k$ for all $k \ge 0$. By proposition \ref{prop:centre-action}, for $k \ge 1$ we also have $z^{-1} \mathcal{N}_k z \subseteq \mathcal{N}_k$, and hence $z^{-1} \mathcal{D}_k z \subseteq \mathcal{D}_k$; and if $z \in \Sigma^{++}$ then the last two inclusions factor through $\mathcal{N}_{k+1}$ and $\mathcal{D}_{k+1}$ respectively.

\begin{remark}
If $G = \GL_n(\Q_p)$, $S$ is the diagonal torus and $P_{\rm min}$ the usual Borel subgroup of upper triangular matrices, then $P$ will be a group of matrices that are block upper triangular (with blocks corresponding to the free roots). Then $T = Z(M)$ is the group of diagonal matrices constant on the blocks of $P$. We may take $\Sigma$ to be the matrices in $T$ whose diagonal entries are powers of $p$, in which case $\Sigma^{+}$ is the monoid where the powers of $p$ increase at each break, and $\Sigma^{++}$ the semigroup where they increase strictly. 
\end{remark}

\begin{definition}\label{def:atkin-lehner-monoid}
Let $\mathbb{I} \subset G(L)$ be the monoid generated by $G_0$ and $\Sigma^+$.
\end{definition}

\begin{theorem}
\begin{enumerate}
 \item The action of $G_0$ on $\mathcal{C}(X, V, k)$ can be extended to an action of $\mathbb{I}$ by continuous $C^\an(X, E)$-linear operators.
 \item If $X$ is reduced, then this action satisfies $\|\gamma \circ f\| \le \|f\|$ for all $\gamma \in \mathbb{I}$ and $f \in \mathcal{C}(X, V, k)$, where $\|\bullet\|$ is the $G_0$-invariant norm defined in proposition \ref{prop:projective}.
\end{enumerate}
\end{theorem}

\begin{proof}
The projection maps $\Psi_{\overline{N}}$, $\Psi_M$ and $\Psi_N$ onto the factors in the Iwahori factorisation are well-defined as elements of the rational function field $L(G)$, and the relation $x = \Psi_{\overline{N}}(x) \Psi_M(x) \Psi_N(x)$ holds in $L(G)$. Evidently the three maps are all regular at any $z \in \Sigma^{+}$, so any element of $\mathbb{I}$ has a (necessarily unique) Iwahori factorisation. Furthermore, if $x \in \mathbb{I}$, then $\Psi_M(x) \in M_0 \Sigma$, since $\Sigma \subset Z(M)$ and $Z(M)$ normalises $N$ and $\overline{N}$; and since conjugation by $\Sigma^+$ preserves $N_0$, we have $\Psi_{N}(x) \in N_0$.

Thus we may define the action of $\mathbb{I}$ on $\mathcal{C}(X, V, k)$ by the same formula as in definition \ref{def:main-construction}, 
\[ (\gamma \circ f)(n, x) = \Delta_X(\Psi_M(n\gamma))(x) \cdot \left[ \Psi_M(n\gamma) \circ f( \Psi_N(n\gamma), x)\right],\]
where $\Delta_X$ is extended to a character of $M_0 \Sigma$ trivial on $\Sigma$.

For $z \in \Sigma^+$, the action of $z$ on $\mathcal{C}(X, V, k)$ is obtained by base extension from an action on $C^\an(\mathcal{D}_k, E)$. This latter space has a canonical supremum norm; and the action of $z$ is the pullback map arising from the inclusion $z^{-1} \mathcal{D}_k z \subseteq \mathcal{D}_k$, and hence is norm-decreasing. Thus the resulting operator on $C^\an(\mathcal{D}_k, E)\ \widehat\otimes_E\ C^\an(X, E)\ \widehat\otimes_E\ V_{M_k-\an}$ is clearly continuous and norm-decreasing for any choice of norm on the latter two factors (since $z$ acts trivially on them), and (ii) now follows from proposition \ref{prop:projective}.
\end{proof}

Note that since $V$ is finite-dimensional, there is some $k$ such that $V_{M_k-\an} = V$. Let $k(V)$ be the least such $k$.

\begin{theorem}\label{thm:linkmaps1} Suppose $z \in \Sigma^{++}$ and $k \ge 1 + k(V)$. Then the action of $z$ defines a continuous and norm-decreasing map $z_k: \mathcal{C}(X, V, k+1) \to \mathcal{C}(X, V, k)$; and if $i_k$ denotes the the obvious inclusion $\mathcal{C}(X, V, k) \into \mathcal{C}(X, V, k+1)$, the diagram 
\[
\xymatrix{
\mathcal{C}(X, V, k) \ar@{^{(}->}[d]^{i_k} \ar[r]^{z_{k-1}} \ar@{.>}[rd]^z  & \mathcal{C}(X, V, k-1) \ar@{^{(}->}[d]^{i_{k-1}} \\
\mathcal{C}(X, V, k+1) \ar[r]^{z_k} & \mathcal{C}(X, V, k)
}
\]
commutes.
\end{theorem}

\begin{proof} Since the action of $z$ is ``constant'' with regard to $X$, it is sufficient to show the corresponding property of the spaces $C^\an(\mathcal{D}_k, V)$; and since $z$ acts trivially on $V$ (and $k-1 \ge k(V)$), we may assume $V = E$. 

By hypothesis, for all $k \ge 1$ we have an inclusion map $\mathfrak{z}_k:z^{-1} \mathcal{D}_k z \into \mathcal{D}_{k+1}$, and hence a pullback map $z_k = \mathfrak{z}_k^*: C^\an(\mathcal{D}_{k+1}, E) \to C^\an(\mathcal{D}_k, E)$. Thus the top and bottom arrows in the diagram are well-defined. Clearly both compositions are equal to the map $f(n) \mapsto f(z^{-1}nz)$; so the diagram commutes, and the diagonal arrow is the action of $z$ on $C^\an(\mathcal{D}_k, E)$, as required.
\end{proof}

When $V$ admits a central character, the action of the split torus $T$ on the space $C^{\an}(\mathcal{N}_k, V)$ can be given explicitly in terms of the basis of proposition \ref{prop:basis}:

\begin{proposition}\label{prop:basis-weights}
Let $j \in J$ and $1 \le j' \le r$ in the notation of proposition \ref{prop:basis}, and suppose that $T_0$ acts on $V$ via the character $\chi: T_0 \to E^\times$. Define
\[s(\underline{j}) = \sum_{\alpha \in \Phi_f} \left(\sum_{1 \le i \le d(\alpha)} \underline{j}(\alpha, i) \right) \alpha,\]
an element of $X^\bullet(T)$. Then the vector $\underline{X}^{\underline{j}} \otimes v_{j'} \in C^{\an}(\mathcal{N}_k, V)$ is an eigenvector for $T_0 \Sigma^{+}$, with
\[ t \circ (\underline{X}^{\underline{j}} \otimes v_{j'}) = \chi(t) t^{-s(\underline{j})} (\underline{X}^{\underline{j}} \otimes v_{j'})\]
for $t \in T_0 \Sigma$, where we extend $\chi$ to a character $T_0 \Sigma^+ \to E^\times$ trivial on $\Sigma^+$.
\end{proposition}

\subsection{Locally algebraic vectors}
\label{ssect:locally-algebraic}

In this section, we shall continue to assume that $V$ is finite-dimensional and endowed with a norm such that $M_0$ acts by isometries; we shall also assume that it is absolutely irreducible and locally algebraic, in the sense of Prasad's appendix to \cite{ST-ugfinite}. Thus we may write $V = V_{alg} \otimes V_{sm}$, where $V_{alg}$ is an absolutely irreducible algebraic representation of $M$, and $V_{sm}$ is an irreducible smooth representation of $M_0$. Note that for such a $V$, we have $k \ge k(V)$ if and only if $V_{sm}$ is trivial on $M_k$; let us choose such a $k$, which will be fixed for the remainder of this section. 

The results in this section are generalisations of results previously obtained (under more restrictive hypotheses) by Owen Jones \cite{owen-transfer}; we give full proofs, as the work of Jones has not yet been published.

\begin{definition}{~}
\begin{itemize}
 \item We define $U_{alg}$ to be the space of polynomial functions $G \to V_{alg}$ satisfying $f(\overline{n} m g) = m \circ f(g)$ for all $\overline{n} \in \overline{N}$, $m \in M$, $g \in G$.
 \item We define $U_{sm}$ to be the induced representation $\Ind_{\overline{P}_0 / \overline{P}_k}^{G_0 / G_k} V_{sm}$.
 \item We define the {\bf classical subrepresentation} $\Ind(V)_{k}^{cl}$ to be $U_{alg} \otimes_E U_{sm}$.
\end{itemize}
\end{definition}

We assume that $V_{alg}$ is dominant and integral for $G$, so $U_{alg}$ is nonzero. Then $\Ind(V)_k^{cl}$ is clearly a nonzero finite-dimensional (and hence closed) $\mathbb{I}$-stable subspace of $\Ind(V)_k$.

As $G_0$ is an open subgroup of $G$, we may differentiate the $G_0$-action on $\Ind(V)_k$ to obtain an action of the Lie algebra $\g = \Lie G$, and thus of its universal enveloping algebra $U(\g)$. We recall the following definition:

\begin{definition}[{\cite[4.1.10]{emerton-memoir}}]
 Let $M$ be an $E$-vector space with an action of $\g$. A vector $v \in M$ is said to be {\it $U(\g)$-finite} if it spans a finite-dimensional subspace of $M$ under the action of $U(\g)$.
\end{definition}

\begin{proposition}
The space $\Ind(V)_k^{cl}$ is exactly the subspace of $U(\g)$-finite vectors in $\Ind(V)_k$.
\end{proposition}

\begin{proof}
It is clear that $\Ind(V)_k^{cl}$ is $U(\g)$-finite, since it is finite-dimensional. Conversely, let $f \in \Ind(V)_k$ be a $U(\g)$-finite vector. We may assume without loss of generality that $f$ is a highest weight vector, since the $U(\g)$-submodule of $\Ind(V)_k^{cl}$ generated by $f$ decomposes as a direct sum of irreducible $U(\g^{ss})$-representations, each generated by highest weight vectors. In this case, applying the isomorphism $\Ind(V)_k \cong C^\an(\mathcal{D}_k, V)$, $f$ is an analytic function on $\mathcal{D}_k$ annihilated by $\n$, so it must be locally constant as a function on $\mathcal{D}_k$. But the set of locally constant analytic functions on $\mathcal{D}_k$ is just $U_{sm}$, so it follows that $f \in \Ind(V)_k^{cl}$.
\end{proof}

Recall that we defined $T = Z(M) \cap S$, a maximal $L$-split torus in $Z(M)$. Since $V$ is irreducible as an $M_0$-representation, $T_0 = T \cap M_0$ must act on $V$ via some character $\chi$. Since $V$ is locally algebraic as an $M_0$-representation and $k \ge k(V)$, the restriction of $\chi$ to $T_k = T \cap M_k$ is algebraic, and thus we can write $\chi = \chi_{sm}\cdot \chi_{alg}$ where $\chi_{alg} \in X^\bullet(T)$ and $\chi_{sm}$ is a character of $T_0 / T_k$. 

Let $\ft = \Lie T$. We now consider the subspace of vectors that are $U(\ft)$-finite, rather than $U(\g)$-finite. This space is $U(\g)$-stable. Let us write $U_{x}$ for the closed subspace of $\Ind(V)_k$ consisting of functions $f \in C^\an(\mathcal{D}_k, V)$ supported on the coset $x \mathcal{N}_k \subseteq \mathcal{D}_k$.

\begin{lemma} \label{lemma:loc-alg-fcns} For each $x \in N_0 / N_k$, the $U(\ft)$-finite vectors in $U_x$ are exactly those that are the restrictions to $x \mathcal{N}_k$ of polynomial functions $f : N \to V$.
\end{lemma}

\begin{proof}
It is sufficient to consider the case $x = 1$, since the space of $U(\ft)$-finite vectors is preserved by $N_0$. Recall that we constructed an orthonormal basis for $C^{\an}(\mathcal{N}_k, V)$ in proposition \ref{prop:basis}. The polynomial functions $E[N] \otimes V$ are exactly the functions in $C^\an(\mathcal{N}_k, V)$ which can be written as a finite linear combination of these basis vectors; since each basis vector is a $\ft$-weight vector by proposition \ref{prop:basis-weights}, any such function is $U(\ft)$-finite.

Conversely, since the above functions are an orthonormal basis, every $f \in C^\an(\mathcal{N}_k, V)$ can be uniquely written $f = \sum_{\underline{j} \in J, 1 \le j' \le r} a_{\underline{j}, j'} \underline{X}^{\underline{j}} \otimes v_{j'}$ for some sequence $a_{\underline{j}, j'} \in E$ tending to zero. For any $t \in T_0 \Sigma^+$, we have
\[ t \circ f = \sum_{\underline{j} \in J, 1 \le j' \le r} \chi(t) t^{-s(\underline{j})} \cdot a_{\underline{j}, j'} \cdot \underline{X}^{\underline{j}} \otimes v_j. \]

Differentiating the above relation, for any $\tau \in \ft$ we have
\[ \tau \circ f = \sum_{\underline{j} \in J, 1 \le j' \le r} [\chi_{alg} - s(\underline{j})](\tau) \cdot a_{\underline{j}, j'} \cdot \underline{X}^{\underline{j}} \otimes v_j.\]
So if $f$ is a $\ft$-weight vector of weight $\gamma$, we deduce that for every $(\underline{j}, j')$ we have $a_{\underline{j}, j'}.(\gamma - \chi_{alg} + s(\underline{j})) = 0$ as elements of $X^\bullet(T)$. But the set of indices $\underline{j}$ for which the bracketed term is zero is finite, and it follows that $f$ is a finite linear combination of the basis vectors, and thus lies in $E[N] \otimes V$. 
\end{proof}

We now recall some standard definitions (first introduced in \cite{bgg-categoryo}):

\begin{definition} Let $R$ be an $E$-representation of a reductive Lie algebra $\g$, $\ft$ a split toral subalgebra of $\g$, and $\p$ a parabolic subalgebra of $\g$ containing $\ft$. 
 \begin{enumerate}
  \item We say $R$ is a $\ft$-weight $\g$-representation if $\ft$ acts locally finitely on $R$, so $R$ is the direct sum of its $\ft$-weight spaces, and each $\ft$-weight space is finite-dimensional.
  \item We say $R$ is in the category $\OO_{\p}$ if it is finitely generated as a $U(\g)$-module, and $\p$ acts locally finitely; this implies that $R$ is a direct sum of finitely many subrepresentations each of which has a highest weight with respect to $\ft$.
 \end{enumerate}
\end{definition}

Any object in the category $\OO_\p$ is automatically a $\ft$-weight $\g$-representation. If $R$ is a $\ft$-weight $\g$-representation, the abstract vector space dual $R^*$ of $R$ is not generally a $\ft$-weight $\g$-representation, but the space of $U(\ft)$-finite vectors in $R^*$ is so; we call this the $\ft^*$-graded dual of $R$ and denote it by $R^\vee$, and it is easy to check that $(R^\vee)^\vee = R$. If $R$ is in $\OO_{\p}$, then $R^\vee$ is in the category $\OO_{\pb}$, where $\pb$ is the opposite parabolic.

\begin{definition}
 Let $V$ be a finite-dimensional representation of $\m$. We define the {\it generalised Verma module} of $V$ (with respect to the parabolic $\pb$) to be the left $U(\g)$-module
\[ \Ver_{\pb}(V) = U(\g) \otimes_{U(\pb)} V.\]
\end{definition}

This is clearly finitely generated as a $U(\g)$-module (as $V$ is finite-dimensional); and it follows from the discussion preceding \cite[prop.~3.1]{lepowsky} that $\pb$ acts locally finitely, so $\Ver_{\pb}(V) \in \OO_{\pb}$ for any finite-dimensional $V$.

\begin{theorem}\label{thm:owen}
The subspace of $U(\ft)$-finite vectors in $\Ind(V)_{k}$ is in the category $\OO_{\p}$, and it is isomorphic as a $\g$-representation to a direct sum of finitely many copies of the dual Verma module $\Ver_{\pb}(V^\vee)^\vee$.
\end{theorem}

\begin{proof} This is a mild generalisation of the main result of \cite{owen-transfer}, and the proof we give is adapted from that reference. Clearly it is sufficient to consider the case $V = V_{alg}$, since as a representation of $G_k$, $\Ind(V)_k$ is isomorphic to a direct sum of copies of $\Ind(V_{alg})_k$.

By lemma \ref{lemma:loc-alg-fcns}, it suffices to show that $E[N] \otimes V$, the space of polynomial functions $N \to V$, endowed with the $U(\g)$-action obtained by differentiating the action of $G_k$ on any of the cosets $x \mathcal{N}_k$, is isomorphic as a $U(\g)$-representation to $\Ver_{\pb}(V^\vee)^\vee$. We identify $E[N] \otimes V$ with the following subspace of $E[C] \otimes V$, where $C$ is the big Bruhat cell $\overline{N} \times M \times N \subset G$:

\begin{definition}
Let $A_P(V)$ be the subspace of $E[C] \otimes_E V$ consisting of polynomial functions $f: C \to V$ such that $f(\overline{p} c) = \overline{p} f(c)$ for all $\overline{p} \in \overline{P}$ and all $c \in C$, where we consider $V$ as a $\overline{P}$-representation by inflation. 
\end{definition}

The map $\g \to \g$ given by multiplication by $-1$ extends to an anti-automorphism of $U(\g)$, the {\it principal anti-automorphism}, which we denote by $S$ (see \cite[\S 2.2.17]{dixmier}).

\begin{lemma}
 There is a bilinear, $\g$-invariant pairing $U(\g)  \times E[C] \to E$ defined by the formula
\[ \langle u, f \rangle = S(u)(f)(1),\]
where $S$ is the principal anti-automorphism of $U(\g)$ (the unique anti-automorphism of $U(\g)$ which acts as multiplication by $-1$ on $\g$).
\end{lemma}

\begin{proof} 
 Since $S(g u) = -S(u)g$ for any $g \in \g$ and $u \in U(\g)$, we have 
\[ \left\langle gu , f \right\rangle + \left\langle u, g f \right\rangle = 
 (S(g u)(f) + S(u)(g f))(1) = 0.\]
as required.
\end{proof}

Hence if $V$ is any finite-dimensional $E$-vector space, we obtain an analogous pairing $(U(\g) \otimes_E V^\vee) \times (E[C] \otimes_E V) \to E$ by defining
\[ \left\langle (u \otimes v^\vee), (f \otimes v) \right\rangle = \langle v^\vee, v\rangle \cdot S(u)(f)(1),\]
which is evidently also $\g$-invariant.

\begin{proposition}
 The orthogonal complement $A_P(V)^\perp \subseteq U(\g) \otimes_E V^\vee$ of $A_P(V) \subseteq E[C] \otimes_E V$ for the above pairing is the submodule $J_P(V)$ generated by elements of the form $u \otimes v^\vee - 1 \otimes u(v^\vee)$ for $u \in U(\pb)$ and $v^\vee \in V^\vee$.
\end{proposition}

\begin{proof}
It is clear that 
\[ \langle x \otimes v^\vee, f\rangle = \langle v^\vee, (S(x) \circ f)(1)\rangle = \langle x v^\vee, f(1)\rangle\]
for all $x \in U(\pb)$ and $f \in A_P(V)$, and hence $J_P(V) \subseteq A_P(V)^\perp$.

Conversely, suppose $x \in A_P(V)^\perp$. As a consequence of the Poincar\'e-Birkhoff-Witt theorem, we may write $x$ as a finite sum of terms of the form $a_1 a_2 \otimes v$ where $a_1 \in U(\n)$ and $a_2 \in U(\pb)$. Since $a_1 a_2 \otimes v - a_1 \otimes a_2 v \in J_P(V)$, we may assume that $a_2 = 1$; so it suffices to prove that $A_P(V)^\perp \cap (U(\n) \otimes V) = 0$.

However, the action of $\n$ on $A_P(V)$ is just the natural action on the first factor under the isomorphism $A_P(V) = E[N] \otimes_E V$, since $\n$ acts trivially on $V$. Since for any nonzero $u \in U(\n)$ we can find some $f \in E[N]$ such that $(uf)(1) \ne 0$, it follows that the annihilator of $A_P(V)$ in $U(\n) \otimes_E V^\vee$ is zero. Hence $J_P(V) = A_P(V)^\perp$ as required.
\end{proof}

\begin{corollary} The space $A_P(V)$ is isomorphic as a $U(\g)$-module to $(U(\g) \otimes_{U(\pb)} V^\vee)^\vee$.
\end{corollary}

\begin{proof}
 The submodule $J_P(V)$ is clearly just the kernel of the natural map $U(\g) \otimes_{E} V^\vee \to U(\g) \otimes_{U(\pb)} V^{\vee}$. Also, the pairing between $(U(\g) \otimes V^\vee)/J_P(V)$ and $A_P(V)$ is nondegenerate, since the annihilator of $A_P(V)$ in the left-hand side is empty by the previous lemma, and it is clear that no nonzero element of $A_P(V)$ is annihilated by every element of $U(\n) \otimes V^\vee$. Since both sides are $\ft$-weight modules, i.e. every $\ft$-weight space is finite-dimensional, the result now follows.
\end{proof}

This completes the proof of theorem \ref{thm:owen}.
\end{proof}

\begin{remark} In fact $(U(\g) \otimes_{U(\pb)} W^\vee)^\vee$ is isomorphic to $(U(\g) \otimes_{U(\p)} W)^{\sigma \vee}$, where $\sigma$ denotes the Chevalley involution which swaps positive and negative root spaces. Both this isomorphism and the definition of the Chevalley involution are non-canonical, depending on a choice of Chevalley basis for $\g$.
\end{remark}

\subsection{Small slopes}
\label{ssect:small-slopes}

We now use the results of the previous section to establish criteria for elements of $\Ind(V)_k$ to lie in $\Ind(V)^{cl}_k$ in terms of the $\Sigma^{++}$-action.

Let us fix an element $\eta \in \Sigma^+$. This defines a linear functional $f_\eta$ on $X^\bullet(T) \otimes_{\Z} \Q$, by $f_\eta(\mu) = -\ord_p(\mu(\eta))$. The minus sign ensures that $f_{\eta}(\alpha) \ge 0$ for every free positive root. 

\begin{definition}\label{def:noncritical}
We say that a weight $\mu \in X^\bullet(T)$ is {\bf strongly classical} if the canonical surjection $\Ver_{\p}(V_{alg}) \to U_{alg}$ is an isomorphism on $\mu$-weight spaces.
We say that $\sigma \in \R_{\ge 0}$ is a {\bf small slope} for $\eta$ if for every character $y \in X^\bullet(T)$ such that $f_\eta(y) \le \sigma$, $\chi_{alg} - y$ is strongly classical, where $\chi_{alg}$ is the central character of $V_{alg}$ as above.
\end{definition}

(The choice of signs is such that there are only finitely many $y \in X^\bullet(T)$ with $f_\eta(y) \le \sigma$ such that the $y$-weight space of $\Ver_{\p}(V_{alg})$ is nonzero.)

\begin{proposition} \label{prop:small-norms}
If $\sigma$ is a small slope for $\eta$, then the endomorphism of the quotient space $\Ind(V)_k/\Ind(V)_k^{cl}$ defined by $\eta$ has norm less than $p^{-\sigma}$.
\end{proposition}

\begin{proof}
Let $\phi \in C^\an(\mathcal{D}_k, V)$, and let $x_0 \in N_0$ be given. Since $\phi$ is analytic on $x_0 \mathcal{N}_k$, we can uniquely write the function $\phi_{x_0}: \xi \mapsto \phi(x_0 \xi)$ on $N_k$ as a convergent series 
\[ \sum_{\underline{j} \in J, 1 \le j' \le r} a_{\underline{j}, j'}(\phi_{x_0}) \underline{X}^{\underline{j}} \otimes v_j\]
in terms of the coordinates introduced above.

Let us apply this with $\eta^{-1} x_0 \eta$ in place of $x_0$; then we can write 
\begin{align*}
(\eta \circ \phi) (x_0 \xi) &= \eta \circ \phi(\eta^{-1} x_0 \eta \cdot \eta^{-1} \xi \eta)\\
&= \sum_{\underline{j} \in J} a_{\underline{j}, j'}(\phi_{\eta^{-1}x_0 \eta}) \underline{X}(\eta^{-1} \xi \eta)^{\underline{j}} \otimes v_{j'}\\
&= \sum_{\underline{j} \in J} a_{\underline{j}, j'}(\phi_{\eta^{-1}x_0 \eta}) \cdot \underline{X}(\xi)^{\underline{j}} \otimes v_{j'} \cdot \eta^{- s(\underline{j})}.
\end{align*}
We can break this sum into two parts according to whether $\ord_p \eta^{-s(\underline{j})} \le \sigma$. The finite sum of terms for which this holds is a linear combination of $\ft$-weight vectors of weights $\chi_{alg} - y$, where $\chi_{alg}$ is the central character of $V_{alg}$ as above, with $y \in X^\bullet(T)$ such that $f_\eta(y) \le \sigma$; this is in $\Ind(V)^{cl}_k$, by hypothesis. The norm of the remaining part is clearly strictly less than $p^{-\sigma} \|\phi\|$, and the result follows.
\end{proof}

\begin{proposition}
If we have
\[ \sigma < \inf_{\alpha \in \Delta_f} f_\eta( s_\alpha(\lambda + \rho) - (\lambda + \rho) ),\]
where $\rho$ is half the sum of the positive roots, and $s_\alpha$ is the reflection in the Weyl group corresponding to $\alpha$, then $\sigma$ is a small slope for $\eta$.
\end{proposition}

\begin{proof} This follows from the first two terms of the generalised Bernstein-Gelfand-Gelfand resolution \cite{lepowsky}: there is an exact sequence of $\g$-representations 
\[ \bigoplus_{\alpha \in \Delta_f} \Ver_\p( s_\alpha(\lambda + \rho) - \rho) \to \Ver_\p(V_{alg}) \to U_{alg} \to 0.\]
Hence the weights of $\Ver_{\p}(V_{alg})$ which are not strongly classical are the weights which also occur in one of the Verma modules $\Ver_{\p}( s_\alpha(\lambda + \rho) - \rho)$. 
\end{proof}

\begin{proposition} If $\lambda = \sum_{\alpha \in \Delta} n_\alpha \lambda_{\alpha}$, where $\lambda_\alpha$ is the fundamental weight corresponding to $\alpha$, then $\epsilon$ is a small slope for $\eta$ if
\[ \epsilon < \inf \{ (n_\alpha + 1) f_\eta(\alpha) \ |\ \alpha \in \Delta_f\}.\]
(Compare \cite[prop 4.7.4]{chenevier-families}.)
\end{proposition}

\begin{proof} An easy calculation shows that $s_\alpha(\lambda + \rho) - (\lambda + \rho) = (1 + n_\alpha) \alpha$ for each $\alpha \in \Delta_f$, and the result follows.
\end{proof}

\section{Classical and overconvergent automorphic forms}
\label{sect:global}

In this section, we shall define overconvergent automorphic forms as functions taking values in the representations constructed in the preceding section. Henceforth $G$ shall be a connected reductive algebraic group over a number field $F$. Let us fix a prime $\p$ of $F$ above $p$, and let $F_{\p}$ be the completion of $F$ at $\p$ and $\OO_{F,\p}$ its ring of integers. We fix a choice of parabolic subgroup $P \subseteq G$ defined over $F_\p$, as in section \ref{sect:local-families}. 

\begin{remark} 
 Any algebraic group over $F$ may be regarded as an algebraic group over $\Q$ by restriction of scalars. However, the spaces we construct for $G$ and for $G' = \Res_{F/\Q} G$ will not be the same, unless $[F_\p : \Q_p] = 1$, since considering $G$ as an algebraic group over $F$ allows us to consider $F_\p$-analytic spaces and functions rather than $\Q_p$-analytic spaces and functions. 
\end{remark}

\subsection{Finiteness of double quotients}
\label{sect:double-quotients}

In this section we recall some classical results on the topology of the ad\`elic points of $G$. Let $\A$ denote the ad\`eles of $F$, $\Af$ the finite ad\`eles, and $F_\infty = F \otimes \R$, so $\A = \Af \times F_\infty$. Let us write $G_\infty$ for $G(F_\infty) = \prod_{v \mid \infty} G(F_v)$. We recall two standard results:

\begin{proposition}[{\cite[theorem 5.1]{platonov-rapinchuk}}]\label{prop:finiteness}
For any compact open subgroup $K \subseteq G(\Af)$, the double quotient
\[ G(F) \backslash G(\Af) / K\]
is finite.
\end{proposition}

\begin{proposition}\label{prop:discreteness}
If $G_\infty$ is compact, then $G(F)$ is discrete in $G(\Af)$.
\end{proposition}

\begin{proof}
See \cite[prop.~1.4]{gross-algebraic} -- this is a special case of the assertion (6) $\Rightarrow$ (3) of that proposition.
\end{proof}

We will assume throughout the remainder of this work that $G$ satisfies the following assumption:

\begin{assumption}\label{compactness-mod-centre}
The quotient $G_\infty / Z(G)(F_\infty)$ is compact. 
\end{assumption}

This is equivalent to assuming that for all infinite places $v$ of $F$, $G(F_v)$ is compact modulo centre. Unless $G$ is abelian, we see that $F$ must be totally real. 

\begin{proposition}\label{finiteness-mod-centre}
 For any compact open subgroup $K \subseteq G(\Af)$, the group $\Gamma'_K = K \cap Z(G)(F)$ has finite index in $\Gamma_K = K \cap G(F)$.
\end{proposition}

\begin{proof}
Let $R$ be the maximal quotient of $G$ which is a torus. The kernel of the quotient map $\sigma: G \to R$ is a semisimple subgroup $G^0 \subseteq G$ defined over $F$, and the natural multiplication map $G^0 \times Z(G) \to G$ is an isogeny.

Assumption \ref{compactness-mod-centre} implies that $G^0(F_\infty)$ is compact, so by proposition \ref{prop:discreteness}, $G^0(F)$ is discrete in $G^0(\Af)$, and hence in $G(\Af)$. Hence for any compact open $K \subseteq G(\Af)$, $\Gamma_K \cap \Ker(\sigma) = G^0(F) \cap K$ is finite. 

Since the natural map $G^0 \times Z(G) \to G$ is an isogeny, and $(\Gamma_K \cap G^0(F)) \times (\Gamma_K \cap Z(G)(F))$ is clearly an arithmetic subgroup of $G^0 \times Z(G)$, \cite[thm.~4.1]{platonov-rapinchuk} shows that its image in $G$ is an arithmetic subgroup of $G$. Hence it has finite index in $\Gamma_K$. Since $\Gamma_K \cap G^0(F)$ is finite, it follows that $\Gamma'_K$ has finite index in $\Gamma_K$.
\end{proof}

\begin{proposition}[{\cite[\S 14.5]{borel-tits}}]
 Let $G_\infty^\circ$ be the connected component of the identity in $G_\infty$. Then $\pi_0 = G_\infty / G_\infty^\circ$ is a finite abelian group.
\end{proposition}

\begin{remark} Gross \cite{gross-algebraic} has studied classical automorphic forms for groups where all arithmetic subgroups are finite; this is equivalent to the assumption that the maximal split torus in $\Res_{F/\Q} Z(G)$ over $\Q$ is a maximal split torus in $G_\infty$, so $G$ is compact at infinity modulo its $\Q$-split centre. This is a strictly stronger condition than compactness modulo centre; for example, the torus $\Res_{K/\Q} \mathbb{G}_m$ for $K$ a real quadratic field has infinite arithmetic subgroups. The methods we use to circumvent the difficulties presented by infinite arithmetic subgroups are based on the arguments used for tori in \cite{buzzard-families}, but with modifications since we consider locally $F_\p$-analytic characters rather than locally $\Q_p$-analytic ones.
\end{remark}

\subsection{Generalities on Hecke algebras}

In this section, we'll recall some standard theory of Hecke algebras. Let $\Gamma$ be a locally compact, totally disconnected topological group. In this subsection, and in \S\S \ref{ssect:module-valued-forms} and \ref{ssect:atkin-lehner}, we can allow the coefficient field $E$ to be any field of characteristic 0 (not necessarily a $p$-adic field). We suppose that $\Gamma$ is unimodular (left and right Haar measure coincide), and we fix a choice of Haar measure $\mu$.

\begin{definition}
The {\bf Hecke algebra} $\HH(\Gamma)$ is the space of compactly supported, locally constant $E$-valued functions on $\Gamma$, endowed with the convolution product 
\[(\phi_1 \star \phi_2)(g) = \int_{h \in \Gamma} \phi_1(h) \phi_2(h^{-1}g)\ \mathrm{d}\mu.\]
\end{definition}

This is an $E$-algebra without identity (unless $\Gamma$ is discrete), but it contains many idempotent elements. In particular, if $K$ is any compact open subgroup, the element $e_K = \mu(K)^{-1} \bbone_K$ is an idempotent. It is convenient to introduce an ordering on idempotents: we say that $e \ge f$ if $ef = fe = f$. Then we see that $e_{K} \ge e_{K'}$ if and only if $K \subseteq K'$. It is easy to see that for any $\phi \in \HH(\Gamma)$, there is some $K$ such that $ e_K \phi = \phi e_K = \phi$, so the directed set $\{e_K\}$ is an {\it approximate identity} for the algebra $\HH(\Gamma)$.

\begin{proposition} \label{prop:monoid-heckealgebra}
If $S$ is any open monoid contained in $\Gamma$, then the compactly supported, locally constant functions on $\Gamma$ with support contained in $S$ form an approximately-unital subalgebra $\HH(S)$ of $\HH(\Gamma)$.
\end{proposition}

\begin{proof} It is sufficient to check that $\HH(S)$ is closed under convolution; but if $\phi_1, \phi_2 \in \HH(S)$, then $\mathop{\rm supp}(\phi_1 \star \phi_2) \subseteq \mathop{\rm supp}(\phi_1) \times \mathop{\rm supp}(\phi_2) \subseteq S$.
\end{proof}

Now let $R$ be any commutative $E$-algebra, and $V$ an $R$-module with a smooth $R$-linear left action of $S$.

\begin{proposition}
\begin{enumerate}
\item There is an action of $\HH(S)$ on $V$ by $R$-linear operators, given by the formula
\[  \phi \circ v = \int_{G} \phi(g) \left( g \circ v \right) {\mathrm d} g.\]
\item For any idempotent $e$, $e V$ is an $R$-submodule of $V$ preserved by the unital subalgebra $e \HH(S) e$.
\item If $e = e_{K}$ for some compact open subgroup $K \subset \Gamma$, then $e V$ is the subspace of $K$-fixed vectors in $V$.
\item If $e = e_1 + e_2$ for orthogonal idempotents $e_1, e_2$ (that is, $e_1 e_2 = e_2 e_1 = 0$), then $e V = e_1 V \oplus e_2 V$.
\item If $e \ge f$ are idempotents, then $e V \supseteq f V$. 
\end{enumerate}
\end{proposition}

\begin{proof} The first three parts are standard (the extension to monoids is immediate). For part (4), if
$e_1$ and $e_2$ are orthogonal, then $e_1 V \cap e_2 V = 0$. Also, $(e_1 + e_2)(e_1 v_1 + e_2 v_2) = e_1 v_1 + e_2 v_2$, so $e V \supseteq e_1 V + e_2 V$. But clearly $e V \subseteq e_1 V + e_2 V$, so the result follows.
It is easy to check that if $e_1 \ge e_2$, then $e_1 - e_2$ is also idempotent and is orthogonal to $e_2$, so (5) is a special case of (4).
\end{proof}

\begin{remark} There are many other interesting idempotents; for example, if $K \subseteq S$ is a compact open subgroup and $\rho$ is a smooth irreducible representation of $K$, then there is an idempotent $e = e(K,\rho) \in \HH(S)$ such that $eV$ is the subspace of $V$ which is $\rho$-isotypical as a $K$-representation. This can be used to construct eigenvarieties interpolating automorphic representations whose local factors at a finite set of primes $\ell \ne p$ lie in given Bernstein components; see \cite[\S 7.3]{bellaiche-chenevier} for further discussion.
\end{remark}

\subsection{Module-valued automorphic forms}
\label{ssect:module-valued-forms}

Let $G$ be an algebraic group satisfying the conditions of \S \ref{sect:double-quotients}. Let us fix a Haar measure on $G(F_\q)$, for each finite place $\q \ne \p$, such that $G(\OO_{F,\q})$ has measure 1.

\begin{definition} We consider the following Hecke algebras:
\begin{itemize}
\item $\HH_f^{(\p)}(G) = \bigotimes_{\q \ne \p}' \HH(G(F_\q))$, where the prime denotes restricted tensor product with respect to the idempotents corresponding to the compact open subgroups $G(\OO_{F,\q}) \subset G(F_\q)$.
\item $\HH^+_{\p}(G)$ is the algebra $\HH(\mathbb{I}) \subset \HH(G(F_\p))$ defined as in proposition \ref{prop:monoid-heckealgebra}, where $\mathbb{I}$ is the monoid of definition \ref{def:atkin-lehner-monoid}.
\item $\HH_{\infty}(G)$ is the algebra of $G_\infty^\circ$-invariant functions on $G_\infty$ (which is isomorphic to the group algebra $E[\pi_0]$ of the finite abelian group $\pi_0$ of components of $G_\infty$).
\end{itemize}
We put $\HH^{(\p)}(G) = \HH_f^{(\p)}(G) \otimes_{E} \HH_\infty(G)$ and $\HH^+(G) = \HH^{(\p)}(G) \otimes_E \HH^+_\p(G)$.
\end{definition}

We can now define automorphic forms with values in a very general class of modules. 

\begin{definition} Let $R$ be any $E$-algebra, and $W$ any $R$-module with an $R$-linear left action of $\mathbb{I}$.

Define $\mathcal{L}(W)$ to be the space of functions $\phi : G(F) \backslash G(\A) \to W$ which satisfy $\phi(gu) = u_\p^{-1} \circ \phi(g)$ for all $u$ in $UG_\infty^\circ$, for some compact open subgroup $U \subset G\left(\Afp\right) \times G_0$ (depending on $\phi$).
\end{definition}

It is clear that $\mathcal{L}(W)$ is an $R$-module; it has a smooth left action of the monoid $G(\A)^+ = G\left(\Afp\right) \times \mathbb{I} \times G_\infty$ via the formula $(\gamma \circ \phi)(g) = \gamma_\p \circ \phi(g\gamma)$, so we may regard it as a left module over $R \otimes_E \HH^+(G)$. In particular, if $e \in \HH^+(G)$ is an idempotent, $e \mathcal{L}(W)$ is a module over $R \otimes_E e\HH^+(G)e$. 

\begin{proposition} \label{prop:invariants}{~}
\begin{enumerate} 
\item If $e_U$ is the idempotent corresponding to $U G_\infty^\circ$ for some compact open subgroup $U \subset G\left(\Afp\right) \times G_0$, $\{\mu_1, \dots, \mu_t\}$ is any set of coset representatives for the finite double quotient
\[ G(F) \backslash G(\A) / U G_\infty^\circ,\]
and 
\[\Gamma_i = UG_\infty^\circ \cap \mu_i^{-1} G(F) \mu_i,\]
then the map 
\[ e_U \mathcal{L}(W) \to \bigoplus_{i=1}^t W^{\Gamma_i}\quad:\quad f\mapsto \left(f(\mu_1), \dots, f(\mu_t)\right)\]
is an isomorphism of $R$-modules. 
 \item If $\alpha \in e_U \mathcal{H}^+(G) e_U$ where $e_U$ is as above, then under the isomorphism $e_U \mathcal{L}(W) \to \bigoplus_{i=1}^t W^{\Gamma_i}$, $\alpha$ corresponds to the element of $\End_{R}\left(\bigoplus_{i=1}^t W^{\Gamma_i}\right)$ given by 
\[(\alpha \circ f)(\mu_j) = \sum_{k=1}^t \sum_{\gamma \in A_{jk}} \alpha(\gamma)\, \left(\gamma_\p \circ f(\mu_k)\right) \]
where $A_{jk} = \left[ \mathop{\rm supp}(\alpha) \cap \mu_j^{-1} G(F) \mu_k\right] / \Gamma_k$.
\end{enumerate}
\end{proposition}

\begin{proof} If $f \in e_U \mathcal{L}(W)$, then $f$ is uniquely determined by the values $f(\mu_i)$. We must have $f(\mu_i) \in W^{\Gamma_i}$, by considering by considering $f(\mu_i g)$ for $g \in \Gamma_i$. Thus the map above is well-defined and injective. Conversely, if $g \in G(\A)$, we may write $g$ in the form $ \gamma \cdot \mu_i \cdot u$ for some $\gamma \in G(F)$, $i \in \{1, \dots, t\}$, and $u \in U G_\infty^\circ$; and $u$ is uniquely determined up to left multiplication by $\Gamma_i$. Hence the map is bijective.

For part (ii), it is sufficient to consider the action of a double coset $U g U G_\infty^\circ$ for some $g \in G(\A)^+$, since these span $e_U \HH^+(G) e_U$ as an $E$-vector space. We see that if $U g U = \bigsqcup_{i=1}^r g_r U$ for some $g_1, \dots, g_r \in G(\A)^+$, then
\[ ([U g U] \circ f)(\mu_j) = \sum_{i = 1}^r (g_i)_\p \circ f(\mu_j g_i).\]

Fix $k \in \{1, \dots, t\}$ and consider the sub-sum consisting of terms for which $\mu_j g_i \in G(F) \mu_k UG_\infty^\circ$. We may choose $g_i$ to lie in $\mu_j^{-1} G(F) \mu_k G_\infty^\circ$, and there will be one such $g_i$ for each coset 
\[\left[U g U \cap \mu_j^{-1} G(F) \mu_k G_\infty^\circ\right] / \Gamma_k.\]
Since $G_\infty^\circ$ acts trivially, this is equivalent to the formula given above.  
\end{proof}

\begin{definition}
We say that $W$ is {\it arithmetical} if there is an arithmetic subgroup of $Z_G$ (or equivalently of $G$, by proposition \ref{finiteness-mod-centre}) which acts trivially on $W$. 
\end{definition}

\begin{corollary} 
For a fixed idempotent $e \in \HH^+(G)$, the functor $W \mapsto e \mathcal{L}(W)$ commutes with base change for arithmetical modules, in the sense that if $R \to S$ is a morphism of commutative $E$-algebras and $W$ is an $R$-module with an arithmetical left action of $\mathbb{I}$, then $\mathcal{L}(W) \otimes_{R} S = \mathcal{L}(W \otimes_{R} S)$.
\end{corollary}

\begin{proof}
For an idempotent of the form $e_U$ as above, the groups $\Gamma_i$ act on $W$ through finite quotients, so it sufficient to show that if $M$ is an $R$-module with an action of a finite group $\Delta$, $(M \otimes_{R} S)^\Delta = M^\Delta \otimes_{R} S$. If $m = \sum_i m_i \otimes s_i \in (M \otimes_{R} S)^\Delta$, then we have 
\[m = \frac{1}{|\Delta|} \sum_{g \in \Delta} g\cdot m = \sum_i \left(\frac{1}{|\Delta|} \sum_{g \in \Delta} g\cdot m_i \right) \otimes s_i,\]
which is clearly in $M^\Delta \otimes_{R} S$.

For a general idempotent $e$, let $U$ be a compact open subgroup of $G(\Afp) \times G_0$ which is sufficiently small that $e_U \ge e$, and $U \cap G(F)$ is contained in $Z(G)$ and acts trivially on $W$. Then part (ii) of prop.~\ref{prop:invariants} shows that there is some integer $t$ and an idempotent $e'$ in the algebra of $t \times t$ matrices over the abstract monoid algebra $E[\mathbb{I}]$ such that $e \mathcal{L}(W)$ is equal to $e' \bigoplus_{i=1}^t W$, and arguing as above we see that this commutes with base change.
\end{proof}

\newcommand{\ep}{e^{(\p)}}
\subsection{The Atkin-Lehner algebra}
\label{ssect:atkin-lehner}

We will be particularly interested in idempotents in $\HH^+(G)$ of the form $e_{G_0} \otimes \ep$, where $\ep$ is an idempotent in $\HH^{(\p)}(G)$. The algebra $e \HH^+(G) e$ then factors as a tensor product $e_{G_0} \HH^+_\p(G) e_{G_0} \otimes \ep \HH^{(\p)}(G) \ep$. The first factor has a very simple structure: 

\begin{lemma}\label{lemma:atkin-lehner-algebra}
Let $\mathcal{A}^+(G)$ be the monoid algebra $E[\Sigma^+]$. Then the map 
\[ \mathcal{A}^+(G) \to e_{G_0} \HH^+_\p(G) e_{G_0} \ :\ z \in \Sigma^+ \mapsto \frac{\gamma(z)^{-1}}{\mu(G_0)} \bbone_{G_0 z G_0}\]  
(extended $E$-linearly to all of $\mathcal{A}^+(G)$) is an isomorphism of $E$-algebras, where $\gamma(z) = | G_0 / (G_0 \cap z G_0 z^{-1})| \in \Z$.
\end{lemma}

\begin{proof} 
First, we show that if $z_1, z_2 \in \Sigma^{+}$, then $[G_0 z_1 G_0] \cdot [G_0 z_2 G_0] = [G_0 z_1 z_2 G_0]$. It is sufficient to check that $z_1 G_0 z_2 \subseteq G_0 z_1 z_2 G_0$. This follows since $z_1 G_0 z_2 = (z_1 \overline{N}_0 z_1^{-1}) \cdot z_1 M_0 z_2 \cdot (z_2^{-1} N_0 z_2)$; $z_1$ and $z_2$ are in $Z(M)$, and $z_1 \overline{N}_0 z_1^{-1}$ and $z_2^{-1} N_0 z_2$ are contained in $G_0$ by the definition of $\Sigma^{+}$. 

The normalisation factor $\gamma(z)$ is chosen so that the image of $z$ is a linear combination of cosets of $G_0$ of total mass 1; hence the map is indeed a homomorphism of rings. Finally, it is clearly bijective, since $\mathbb{I}$ is generated by $G_0$ and $\Sigma^{+}$, and hence the cosets $G_0 z G_0$ for $z \in \Sigma^+$ span $e_{G_0} \HH^+_\p(G) e_{G_0}$ as an $E$-vector space.
\end{proof}

\begin{lemma}\label{lemma:finiteslope1} If $R$ is a commutative $E$-algebra and $\lambda$ is a ring homomorphism $\mathcal{A}^+(G) \to R$, and there exists $\eta \in \Sigma^{++}$ such that $\lambda(\eta) \in R^\times$, then $\lambda(\delta) \in R^\times$ for all $\delta \in \Sigma^{++}$.
\end{lemma}

\begin{proof} Any two elements $x, y \in \Sigma^{++}$ are {\it commensurable}, in the sense that there exist $\alpha, \beta \in \Sigma^{++}$ and $m, n \ge 1$ such that $\alpha x = y^m$ and $\beta y = x^n$. Thus if $\lambda(x)$ is a non-unit, $\lambda(y^m)$ must be a non-unit and hence $\lambda(y)$ cannot be a unit either.
\end{proof}

If $W$ is an $R$-module with an action of $\mathbb{I}$, and $V \subseteq e_{G_0} \mathcal{L}(W)$ is a $\mathcal{A}^+(G)$-stable finite rank projective $R$-submodule, it follows that if $[G_0 z G_0]$ is invertible on $V$ for one $z \in \Sigma^{++}$, then this holds for all $z \in \Sigma^{++}$. If this holds, we say $V$ is {\it finite slope}.

\subsection{Property (Pr)}

We now resume our running assumption that $E$ is a discretely valued extension of $\Q_p$, and we take $R$ to be a Banach algebra over $E$ and $W$ a Banach $R$-module. We suppose that $\mathbb{I}$ acts on $W$ via continuous norm-decreasing operators, with the subgroup $G_0 \subset \mathbb{I}$ acting by isometries. It is clear from proposition \ref{prop:invariants} that every element of $\mathcal{L}(W)$ has a well-defined supremum norm, and that $\HH^+(G)$ acts on $\mathcal{L}(W)$ by continuous operators.

Let us recall the following definition from \cite{buzzard-eigen}. If $M$ is a Banach module over the nonarchimedean Banach algebra $A$, and there is some other $A$-Banach module $N$ such that $M \oplus N$ is isomorphic (but not necessarily isometrically isomorphic) to an orthonormalisable $A$-Banach module, we say $M$ {\bf has property (Pr)}. 

\begin{remark} 
The notation (Pr) is intended to suggest projective modules, but such modules are not necessarily projective as abstract $A$-modules, and the category of Banach $A$-modules is not abelian (since there exist non-surjective morphisms of Banach $A$-modules with dense image). However, the universal property given in \cite{buzzard-eigen} shows that the Banach modules with property (Pr) are relative projective modules in the sense of \cite[def.~1.6]{taylor-relative}.
\end{remark}

\begin{proposition}\label{prop:projective2}
For any idempotent $e \in \HH^+(G)$, $e \mathcal{L}(W)$ is a Banach $R$-module. If $W$ is arithmetical and has property (Pr), then so does $e \mathcal{L}(W)$.
\end{proposition}

\begin{proof} 
Let $U$ be a compact open subgroup sufficiently small that $e_U \ge e$. Applying prop.~\ref{prop:invariants}(ii), we see that $e \mathcal{L}(W)$ is the kernel of an element of $M_{t \times t}(E[\mathbb{I}])$ acting on $W^{\oplus t}$ for some $t \in \Z_{\ge 1}$; in particular, it is a closed $R$-submodule of $W^{\oplus t}$, so it is a Banach $R$-module. 

If $W$ is arithmetical, we can choose $U$ sufficiently small that all of the subgroups $\Gamma_i$ of proposition \ref{prop:invariants} act trivially on $W$; so $e_U \mathcal{L}(W)$ is a direct sum of copies of $W$, which has property (Pr) if $W$ does. Now, since $e_U \ge e$, $e$ and $e_U - e$ are orthogonal idempotents, so we have $e_U \mathcal{L}(W) = e \mathcal{L}(W) \oplus (e_U - e)\mathcal{L}(W)$. Both of these subspaces are closed; so if $W$ has property (Pr) then $e\mathcal{L}(W)$ is a direct summand of a module which has property (Pr), and hence it also has property (Pr). 
\end{proof}

\subsection{Arithmetical weights}

In order to apply the above machinery in the case where $W$ is one of the representations constructed in the previous chapter, we must establish under what conditions such representations are arithmetical. Let $V$ be a finite-dimensional locally $L$-analytic representation of $M_0$, $X$ an affinoid subset of $\widehat H_0$ defined over $E$, and $k \in \Z_{\ge 1}$ such that $k \ge \max(k(X), k(V))$. The following is immediate from the definitions:

\begin{proposition}
 If both $V$ and the universal character $\Delta_X : H_0 \to C^\an(X, E)^\times$ are arithmetical, then $\mathcal{C}(X, V, k)$ is arithmetical.
\end{proposition}

We now study the subspace of the local weight space $\widehat H_0$ consisting of arithmetical weights. We need the following purely local construction:

\begin{proposition}
Let $L/\Q_p$ be a finite extension; $T$ a torus over $L$; and $\Gamma$ a finitely generated subgroup of $T(L)$. Then for any compact open subgroup $U \subseteq T(L)$, the functor mapping an affinoid algebra $A$ over $L$ to the abelian group of locally $L$-analytic characters $U \to A^\times$ which are trivial on a finite index subgroup of $\Gamma \cap U$ is represented by a closed rigid analytic subgroup $\widehat U^\Gamma \subseteq \widehat U$ over $L$; and there is an integer $d$ satisfying $\dim T \ge d \ge \dim T - \rank (T \cap \Gamma)$ such that each component of $\widehat U^\Gamma$ is isomorphic over $\C_p$ to a finite \'etale cover of $\mathbf{B}^\circ(1)^d$.
\end{proposition}

\begin{proof} Firstly, let us assume that $U$ is a good $L$-analytic open subgroup. Without loss of generality we assume $\Gamma \subseteq U$. Let $g_1, \dots, g_r$ be a set of generators for $\Gamma$. Then for any $L$-affinoid algebra $A$, the locally $L$-analytic characters $U \to A^\times$ which satisfy the above condition are precisely those characters $\kappa$ for which $\kappa(1), \dots, \kappa(r)$ are all roots of unity in $A$ -- necessarily $p$-power roots of unity, since $\kappa$ is continuous and $U$ is pro-$p$. Thus they are exactly the characters for which the $r$ analytic functions $\kappa \to \log \kappa(g_i)$ vanish. These functions cut out a Zariski-closed subvariety of $\widehat U$, which we denote by $\widehat U^\Gamma$.

To determine the dimension of $\widehat U^\Gamma$, we give an alternative description. Since $U$ is a good subgroup, we may equate it with an $\OO_L$-lattice $\mathfrak{u} \subset \Lie T$ in the usual way. Then if $\tilde \Gamma$ denotes the $\OO_L$-span of the $g_i$, $\tilde \Gamma$ is an $\OO_L$-submodule of $U$, and hence is closed. Let $d = \rank_{\OO_L} \mathfrak{u} - \rank_{\OO_L} \tilde \Gamma$. Since $0 \le \rank_{\OO_L} \tilde \Gamma \le r$, we have $\dim T - r \le d \le \dim T$.

Let $U_m = \mathfrak{u} / \pi_L^m \tilde \Gamma$. For all $m$ this is a compact locally $L$-analytic group of dimension $d$, and the natural map $\widehat U_m \to \widehat U$ is a closed immersion, as $\widehat U_m$ is exactly the locus where the evaluation maps $\widehat U \to \mathbb{G}_m$ corresponding to a finite set of topological generators of $\pi_L^m \tilde \Gamma$ are equal to 1. Each component of $\widehat U_m$ is isomorphic over $\C_p$ to a finite \'etale cover of $\mathbf{B}^\circ(1)^d$; and the obvious map $\widehat U_m \to \widehat U_{m+1}$ identifies $\widehat U_m$ with a union of components of $\widehat U_{m+1}$.

I claim that the union of all of the $\widehat U_m$ is $\widehat U^\Gamma$; this is equivalent to the statement that any $L$-analytic character $U \to A^\times$ trivial on a finite index subgroup of $\Gamma$ is in fact trivial on $\pi_L^m \tilde \Gamma$ for some sufficiently large $m$. This follows from the fact that for any such character $\chi$, we can find an $m$ such that $m$ extends to a rigid analytic function on the affinoid ball associated to $\pi_L^m \mathfrak{u}$, and is also trivial on $\Gamma_m = \Gamma \cap \pi_L^m \mathfrak{u}$; and the Zariski closure of $\Gamma_m$ in this ball is $\pi_L^m \tilde \Gamma$. Hence $\widehat U^{\Gamma}$ is equidimensional, with every component isomorphic over $\C_p$ to a finite \'etale cover of $\mathbf{B}^0(1)^d$.

If $U$ is not a good $L$-analytic open subgroup, then we can find some $U' \subseteq U$ of finite index which is so; then defining $\widehat U^{\Gamma}$ to be the inverse image of $\widehat U'^{\Gamma}$ in $\widehat U$, the result follows, since $U' \cap \Gamma$ has finite index in $U \cap \Gamma$. 
\end{proof} 

Taking $T$ to be $H$ and $\Gamma$ to be the image of $Z(G)(\OO_F)$ in $H(\OO_{F,\p})$, this gives a rigid space over $F_\p$ parametrising the arithmetical locally $F_\p$-analytic characters of $H_0$. We write $\widehat H_0^{arith}$ for this space.

\begin{remark}
The codimension of $\widehat H_0^{arith}$ in $\widehat H_0$ is a subtle quantity. It is independent of the choice of $P$, and depends only on the torus $Z(G)$; it is clearly 0 if $Z(G)(\OO_F)$ is finite, and $\ge 1$ if it is infinite. If $F = \Q$, then the sublattice $\tilde \Gamma$ above is just the topological closure of $\Gamma$ in $U$, and we see that the assertion $\mathop{\mathrm{codim}} \widehat H_0^{arith} = \rank Z(G)(\Z)$ is equivalent to the Leopoldt conjecture for tori of \cite[\S 4.3.3]{hill-banach} (where it is shown that this is equivalent to the conventional Leopoldt conjecture).
\end{remark}

\subsection{Overconvergent automorphic forms}

\begin{definition}\label{firstMr}
Let $e \in \HH^{+}(G)$ be an idempotent, $X$ an affinoid subset of $\widehat H_0^{arith}$, and $k \ge \max(k(V), k(X))$. We define the space of $k$-overconvergent automorphic forms of type $e$ and weight $(X,V)$ as the space 
\[ M(e, X, V, k) = e \mathcal{L}\bigg(\mathcal{C}(X, V, k)\bigg).\]
\end{definition}

We now collect together some properties of these spaces, all of which are immediate from properties we have established for the spaces $\mathcal{C}(X, V, k)$.

\begin{theorem}\label{thm:spacesM}
 The spaces $M(e, X, V, k)$ have the following properties:
\begin{enumerate}

 \item $M(e, X, V, k)$ is a Banach module over $C^\an(X, E)$ with property $(Pr)$, endowed with an action of $e\HH^+(G)e$ by continuous $C^\an(X, E)$-linear operators. 

\item The maps $\mathcal{C}(X, V, k) \to \mathcal{C}(X, V, k+1)$ induce maps $i_k: M(e, X, V, k) \to M(e, X, V, k+1)$ which are $C^\an(X, E)$-linear, $e\HH^+(G)e$-equivariant, injective, and have dense image.

\item The construction commutes with base change of reduced affinoids, in the sense that if $Y \subseteq X$ and $k \ge \max(k(X), k(V))$, then
\[ M(e, X, V, k) \widehat\otimes_{C^\an(X, E)} C^\an(Y, E) = M(e, Y, V, k)\]
as a Banach $C^\an(Y, E)$-module, and this isomorphism is compatible with the $e\HH^+(G)e$-action.

\item \label{item:automorphic-compactness} If $T \in e\HH^+(G)e$ is supported in $G_0 \Sigma^{++} G_0$, and $k \in \Z_{\ge 1}$ is such that $k \ge \max(k(X), k(V))$, then the action of $T$ on $M(e,X,V,k+1)$ in fact defines a continuous map $T_k: M(e, X, V, k+1) \to M(e, X, V, k)$; and if $k \ge \max(k(X), k(V)) + 1$, we have a commutative diagram
\[
\xymatrix{
M(e, X, V, k) \ar@{^{(}->}[d]^{i_k} \ar[r]^{T_{k-1}} \ar@{.>}[rd]^T  & M(e, X, V, k-1) \ar@{^{(}->}[d]^{i_{k-1}} \\
M(e, X, V, k+1) \ar[r]^{T_k} & M(e, X, V, k)
}
\]
where $i_k$ ias as above. In particular, the endomorphism of $M(e, X, V, k)$ defined by $T$ is compact.
\end{enumerate}
\end{theorem}

From part \ref{item:automorphic-compactness} of the above theorem, we can deduce the following corollary, which is analogous (both in statement and proof) to \cite[thm.~5.2]{buzzard-continuation}.

\begin{corollary}\label{corr:automorphic-maximally-overconvergent} Let $T$ be as above, and let $\lambda \in E^\times$ with $\lambda \ne 0$. Then the natural map $M(e, X, V, k) \to M(e, X, V, k+1)$ restricts to an isomorphism on the generalised $\lambda$-eigenspaces for $T$.
\end{corollary}

\begin{proof}
Let us write $M(e, X, V, k)_{\lambda}$ for the generalised $\lambda$-eigenspace. Since $T$ is compact, any generalised eigenspace for $T$ is finite-dimensional, so we may suppose that $M(e, X, V, k+1)_{\lambda} = \Ker (T - \lambda)^m$ for some $m$. 

Clearly $i_k$ restricts to an injection $M(e, X, V, k)_{\lambda} \to M(e, X, V, k+1)_{\lambda}$. Let $T'$ be the composition of the endomorphism $-\sum_{i=1}^{m} \binom{m}{i} \frac{T^{i-1}}{(-\lambda)^{i}}$ of $M(e, X, V, k+1)$ and the map $T_k: M(e, X, V, k+1) \to M(e, X, V, k)$ above. Then $T'$ gives a map $M(e, X, V, k+1)_{\lambda} \to M(e, X, V, k)_{\lambda}$ which is a section of $i_k$.
\end{proof}

\subsection{Relation to classical automorphic forms}

We now study the relation between the spaces $\mathcal{L}(W)$ and classical automorphic forms when $W$ is locally algebraic, using methods based on \cite[\S 8]{gross-algebraic}.

Fix an embedding $\iota: E \into \C$. This determines an embedding $F \into \C$, since $F \subset F_\p \subseteq E$. If $W_{alg}$ is an algebraic representation of $G$ over $E$, then the choice of $\iota$ determines an extension of $W$ to an algebraic representation of $G_\infty = \prod_{v \mid \infty} G(F_v)$, with all factors acting trivially except that corresponding to $\iota$. We denote this representation by $\iota(W)$.

\begin{proposition}\label{prop:classical-automorphic}
Suppose $W$ is a finite-dimensional irreducible locally algebraic representation of $\mathbb{I}$, of the form $W_{sm} \otimes W_{alg}$ where the factors are respectively smooth and algebraic representations. Fix an embedding $\iota: E \into \C$. Then $\mathcal{L}(W)\,\otimes_{E, \iota}\,\C$ is isomorphic as a $G(\Af) \times \pi_0$-representation to 
\[ \bigoplus_{\pi} m(\pi) \left[\left( \sideset{}{'}\bigotimes_{\q \ne \p \ \text{finite}} \pi_\q\right) \otimes (W_{sm} \otimes \pi_\p) \otimes (\iota(W_{alg}) \otimes \pi_\infty)^{G_\infty^\circ}\right],\]
where the direct sum is over all automorphic representations $\pi$ of $G$ for which $(\iota(W_{alg}) \otimes \pi_\infty)^{G_\infty^\circ}$ is nonzero, and $m(\pi)$ is the multiplicity of $\pi$ in the space of automorphic forms for $G$.
\end{proposition}

\begin{proof}
Let $\rho_1$ denote the action of $G_0$ on $W_{sm} \otimes W_{alg}$ obtained from the first factor alone, and $\rho_2$ the action from the second factor. These two actions commute, and the standard action $\rho_{1,2}$ on the tensor product is the composition of both. 

Given $f \in \mathcal{L}(W)$, let us write $f_\infty$ for the function $G(\A) \to W \otimes_E \C$ given by $f_\infty(g) = \rho_2(g_\infty^{-1} \cdot \iota(g_\p)) f(g)$. This is well-defined, since $\rho_2$ is algebraic. Furthermore, if $\gamma \in G(F)$, then $\gamma_\infty^{-1} \cdot \iota(\gamma_\p)$ acts trivially on $\iota(W_{alg})$; hence $f_\infty$ gives a function on $G(F) \backslash G(\A)$. 

Since $f \in \mathcal{L}(W)$, there is some compact open subgroup $U \subseteq G_0 \times G(\Afp)$ such that $f(gu) = \rho_{1,2}(u_\p^{-1}) f(g)$ for all $u \in UG_\infty^\circ$. A calculation shows that this corresponds to the relation $f_\infty(gu) = \rho_1(u_\p)^{-1} \rho_2(u_\infty)^{-1} f_\infty(g)$ for all $u$ in this subgroup. 

We deduce that for any linear functional $\phi$ on $W \otimes_E \C$, $\phi \circ f$ is an automorphic form in the sense of \cite[def.~2.2]{cogdell-et-al}: a function $G(F) \backslash G(\A)$ which is locally constant on $G(\Af)$ (as $W_{sm}$ is smooth), is $G_\infty^\circ$-finite and $U(\g_\C)$-finite (as $W_{alg}$ is finite dimensional), and has polynomial growth (since $W_{alg}$ is algebraic).

Hence $f_\infty$ lies in the space 
\[\left(\mathcal{A}(G) \otimes \iota(W_{alg}) \otimes W_{sm}\right)^{G_\infty^\circ},\]
where $\mathcal{A}(G)$ denotes the space of automorphic forms for $G$ in the sense of \cite[def.~2.2]{cogdell-et-al}, and $G(\A)$ acts by translation on the first factor, via the chosen factor of $G_\infty$ on the second, and via $G(F_\p)$ on the third. Conversely, any function $f_\infty$ in this space clearly gives an element of $\mathcal{L}(W)$ by reversing the construction.

Since $G$ has no cusps, all automorphic forms for $G$ are cuspidal. Thus for any smooth character $\omega: Z(G)(\A) / Z(G)(F) \to \C^\times$, the space of automorphic forms of central character $\omega$ decomposes as a direct sum of automorphic representations \cite[\S 3.4]{cogdell-et-al}. For any compact open $U \subseteq G(\Af)$, there are only finitely many characters $\omega: Z(G)(\A) / Z(G)(F) \to \C^\times$ that are trivial on $Z(G)(\Af) \cap U$ and coincide with the central character of $\iota(W_{alg})^\vee$ on $G_\infty^\circ$, and hence $e_U \mathcal{L}(W)$ can be written as a direct sum of terms arising from automorphic representations. Passing to a direct limit over $U$, we obtain the result.
\end{proof}

\begin{definition}
Let $\pi_\infty$ be an irreducible continuous representation of $G_\infty$ on a finite-dimensional $\C$-vector space. We say $\pi_\infty$ is {\it allowable} if there is an an algebraic representation $W$ of $G$ over $E$ such that $(\iota(W) \otimes_\C \pi_\infty)^{G_\infty^\circ} \ne 0$. We say an automorphic representation $\pi = \sideset{}{'}{\bigotimes_{v}} \pi_v$ is allowable if its factor $\pi_\infty = \bigotimes_{v \mid \infty} \pi_v$ is so.
\end{definition}

\begin{remark}
 The choice of embedding $\iota$ determines a map $F_\infty \to \C$, and hence a map $G_\infty \to G(\C)$, and the algebraic representations of $G_\infty$ which are of the form $\iota(W)$, for $W$ an algebraic representation of $G$ over $E$, are exactly those which arise by inflation from an algebraic representation of $G(\C)$. Hence our definition of ``allowable'' coincides with that of \cite[def.~3.1.3]{emerton-interpolation}.
\end{remark}

\subsection{The classical subspace}

We will now investigate the classical subspaces of the fibres of $M(e, X, V, k)$ at points of $X$ corresponding to locally algebraic characters. By proposition \ref{prop:weight-fiddle}, it is sufficient to take $X = \bbone$, the point of $\widehat H_0^{arith}$ corresponding to the trivial character, and study $M(e, \bbone, V, k) = e \mathcal{L}(\Ind(V)_k)$. As before, we suppose that $V = V_{alg} \otimes V_{sm}$ is a locally algebraic representation of $M_0$, and that $k \ge k(V)$; let $U_{sm, k}$ and $U_{alg}$ be the corresponding smooth and algebraic parabolically induced representations of $G_0$, so $\Ind(V)_k^{cl} = U_{sm, k} \otimes U_{alg} \subseteq \Ind(V)_k$.

\begin{definition}
 Let $M(e, \bbone, V, k)^{cl} = e \mathcal{L}(\Ind(V)_k^{cl})$
\end{definition}

This is a finite-dimensional (and hence closed) subspace of the $E$-Banach space $M(e, 0, V, k) = e \mathcal{L}(\Ind(V)_k)$.

\begin{theorem}\label{thm:classical-automorphic}
The space $M(e, 0, V, k)^{cl}$ is isomorphic as an $e \HH^+(G) e$-module to
\[ \bigoplus_\pi m(\pi) \cdot e \left[\left( \sideset{}{'}\bigotimes_{\q \ne \p \ \text{finite}} \pi_\q\right) \otimes (U_{sm} \otimes \pi_\p) \otimes (\iota(U_{alg}) \otimes \pi_\infty)^{G_\infty^\circ}\right]. \]
where the direct sum is over all automorphic representations $\pi$ for which $(\iota(U_{alg}) \otimes \pi_\infty)^{G_\infty^\circ} \ne 0$.
\end{theorem}

\begin{proof}
Immediate from proposition \ref{prop:classical-automorphic}.
\end{proof}

We now show that the presence of $U_{sm}$ in the above formula allows us to restrict to idempotents of the form $\ep \otimes e_{G_0}$, where $\ep$ is an idempotent away from $\p$:

\begin{lemma} If $\pi_\p$ is an irreducible smooth representation of $G(F_\p)$ which is a subquotient of a representation obtained by parabolic induction from $\overline{P}$, then there is some $k \ge 0$ and some finite-dimensional smooth representation $V_{sm}$ of $M_0$ (trivial on $M_k$) such that 
\[ \left(U_{sm, k} \otimes \pi_\p\right)^{G_0} \ne 0,\]
where $U_{sm, k} = \Ind_{\overline{P}_0/\overline{P}_k}^{G_0 / G_k} V_{sm}$ as above.
\end{lemma}

\begin{proof}
The assumption on $W$ is equivalent to assuming that the Jacquet module $W_{\overline{N}} \ne 0$, by a theorem of Jacquet -- see \cite[6.3.7]{casselman-book}. By [{\it op.cit.}, 3.3.1], $W_{\overline{N}}$ is an admissible smooth $M$-representation. Hence $(W_{\overline{N}})^{M_k}$ is finite-dimensional, and nonzero for sufficiently large $k$. 

The space $W^{G_k}$ naturally has an action of $e_{G_k} \HH^+_\p(G) e_{G_k}$, and we may identify the monoid algebra $E[\Sigma^{+}]$ with a subalgebra of this algebra (as was done for $G_0$ in \S \ref{ssect:atkin-lehner} above). Casselman's theory of the canonical lifting \cite[\S 4.1]{casselman-book} shows that if $W^{G_k}_{fs}$ denotes the maximal finite slope subspace of $W^{G_k}$, then the quotient map $W \to W_{\overline{N}}$ gives an isomorphism $W^{G_k}_{fs} \stackrel{\sim}{\to} (W_{\overline{N}})^{M_k}$. Thus this finite slope subspace is also nonzero for large enough $k$.

The space $W^{G_k}$ naturally has an action of $G_0$, since $G_k$ is normal in $G_0$; and any element of $W^{G_k}_{fs}$ is in fact $\overline{N}_0$-invariant, since it is clearly $\overline{N}_k$-invariant and we may conjugate $\overline{N}_0$ into $\overline{N}_k$ via an element of $\Sigma^{++}$. Thus the $\overline{N}_0$-invariants of the $G_0$-representation $W^{G_k}$ are not zero. 

Let $V = (W^{G_k})^{\overline{N}_0}$, which is a representation of $M_0 / M_k$. Then we have
\begin{align*}
 &\Hom_{M_0 / M_k} \left(V, (W^{G_k})^{\overline{N}_0}\right) \ne 0\\
 \Rightarrow& \Hom_{\overline{P}_0 / \overline{P}_k} \left(V, W^{G_k}\right) \ne 0\\
 \Rightarrow& \Hom_{G_0 / G_k} \left( \Ind_{\overline{P}_0 / \overline{P}_k}^{G_0/G_k} V, W^{G_k} \right) \ne 0\\
 \Rightarrow& \Hom_{G_0} \left( \Ind_{\overline{P}_0 / \overline{P}_k}^{G_0/G_k} V, W \right) \ne 0\\
 \Rightarrow& \left[\left(\Ind_{\overline{P}_0 / \overline{P}_k}^{G_0/G_k} V^*\right) \otimes W \right]^{G_0} \ne 0.
\end{align*}

So we may take $V_{sm} = V^*$.
\end{proof}

\begin{corollary}\label{corr:wild-level-lowering}
Let $\pi = \sideset{}{'}{\bigotimes_v} \pi_v$ be an allowable automorphic representation of $G$ for which $\pi_\p$ is isomorphic to a subquotient of a representation parabolically induced from $P$. Let $\ep_f$ be an idempotent in $\HH_f^{(\p)}(G)$ such that $\ep_f \pi_f^{(\p)} \ne 0$. 

Then there exists a locally algebraic $M_0$-representation $V$ and, for all $k \ge k(V)$, a nonzero $\HH^{+}(G)$-invariant finite-slope subspace of $M(\ep_f \otimes e_{G_0}, \bbone, V, k)^{cl}$ which is isomorphic as an $\ep_f \HH_f^{(\p)}(G) \ep_f$-module to a direct sum of copies of $\ep_f \pi_f^{(\p)}$.
\end{corollary}

\begin{proof}
 By assumption, there is an algebraic representation $W_{alg}$ of $G$ such that $(\iota(W_{alg}) \otimes \pi_\infty)^{G_\infty^\circ} \ne 0$; this is necessarily equal to the induced representation $U_{alg}$ arising from some algebraic representation $V_{alg}$ of $M$, by highest weight theory. The previous lemma shows that we can find a smooth representation $V_{sm}$ of $M_0$ such that $e_{G_0} \left(U_{sm, k} \otimes \pi_\p\right) \ne 0$ for large enough $k$. Letting $V = V_{sm} \otimes V_{alg}$ and applying theorem \ref{thm:classical-automorphic}, $M(\ep \otimes e_{G_0}, 0, V, k)^{cl}$ contains a direct summand isomorphic as a $\HH^+(G)$-representation to $\ep_f \pi_f^{(\p)} \otimes e_{G_0} \left(U_{sm, k} \otimes W\right) \otimes (\iota(W) \otimes \pi_\infty)^{G_\infty^\circ}$.
\end{proof}

\begin{remark}
These are not the only automorphic representations that can contribute to $M(\ep_f \otimes e_{G_0}, \bbone, V, k)^{cl}$; since $G_0$ could be arbitrarily small, some supercuspidal $G(F_\p)$-representations can have $G_0$-invariant vectors, or more generally $G_0$-stable subspaces isomorphic to a parabolically induced representation of $G_0$. Thus automorphic representations with supercuspidal local factors at $\p$ can still contribute to $M(\ep_f \otimes e_{G_0}, \bbone, V, k)^{cl}$ for suitable $V$. However, these will not be finite slope.
\end{remark}

Let us choose an idempotent $e$ of the form $\ep \otimes e_{G_0}$, as above. The next theorem shows that any eigenform in $M(e, \bbone, V, k)$ of sufficiently small slope lies in $M(e, \bbone, V, k)^{cl}$, giving an analogue in our situation of the control theorem of \cite{C-CO} for overconvergent modular forms. Let us fix an $\eta \in \Sigma^{++}$, and $\lambda \in \overline{E}^\times$.

\begin{theorem}\label{thm:classicality}
If $\sigma = \ord_p(\lambda)$ is a small slope relative to $\eta$ and $V$ (in the sense of definition \ref{def:noncritical}), then the generalised $\lambda$-eigenspace of $[G_0 \eta G_0]$ acting on $M(e, \bbone, V, k)$ is contained in $M(e, \bbone, V, k)^{cl}$.
\end{theorem}

\begin{proof} 
From proposition \ref{prop:small-norms} we know that the operator norm of $[G_0 \eta G_0]$ acting on the quotient space $Q = M(e, \bbone, V, k)/M(e, \bbone, V, k)^{cl}$ is strictly less than $p^{-\sigma}$. Hence the image in $Q$ of the generalised $\lambda$-eigenspace of $M(e, \bbone, V, k)$ must be zero, as otherwise the operator $[G_0 \eta G_0]$ acting on $Q$ would have an eigenvalue larger than its norm. 
\end{proof}

\subsection{Relation to Emerton's completed cohomology}
\label{ssect:emerton}

In this subsection we will investigate the relation between the spaces $M(e, X, V, k)$ and the constructions of \cite{emerton-interpolation}. As in the discussion following Emerton's definition 3.2.3, when $G$ satisfies assumption \ref{compactness-mod-centre}, the space $\widetilde H^0_{F_\p-\la}$ is the space of continuous $E$-valued functions on $G(\A)$ which are left $G(F)$-invariant, locally constant on cosets of $G(\Afp) \times G_\infty^\circ$, and locally $F_\p$-analytic on cosets of $G(F_\p)$. (In the reference cited, there is a stronger running assumption that $\Res_{F/\Q} G$ is compact at infinity modulo $\Q$-split centre; but this is not used until the end of the section.)

Assume $e$ is an idempotent in $\HH^+(G)$ of the form $\ep \otimes e_{G_0}$, as in the previous section. By proposition \ref{prop:weight-fiddle} it is sufficient to study the spaces $M(e, \bbone, V, k)$ where $(V, \rho)$ is an irreducible locally $F_\p$-analytic representation of $M_0$. We regard $V$ as a representation of $\overline{P}_0 = \overline{N}_0 M_0$, trivial on $\overline{N}_0$.

\begin{proposition}
There is an isomorphism
\[ \varinjlim_k M(e, \bbone, V, k) \stackrel{\sim}{\to} \ep \left(\widetilde H^0_{F_\p-\la} \otimes_E V \right)^{\overline{P}_0}, \]
commuting with the action of $\ep \HH^{(\p)}(G) \ep$ on both sides.
\end{proposition}

\begin{proof}
Since any $f \in \varinjlim_k M(e, \bbone, V, k)$ is determined by finitely many values $f(\mu_1), \dots, f(\mu_t)$, it follows that the functor $W \mapsto e \mathcal{L}(W)$ commutes with direct limits; hence
\[\varinjlim_k M(e, \bbone, V, k) = e \cdot \mathcal{L}(\Ind_{\overline{P}}^{G_0} V),\]
where $\Ind_{\overline{P}}^{G_0} V = C^{\la}(N_0, V)$, the space of locally $F_\p$-analytic $V$-valued functions on $N_0$.

Given $f \in \varinjlim_k M(e, \bbone, V, k)$, we can thus regard $f$ as a map $G(\A) \to C^{\la}(N_0, V)$. Let us take $\tilde f$ to be the function $G(\A) \to V$ given by $\tilde f(g) = f(g)(1)$. Then $\tilde f$ is evidently in $\widetilde H^0_{F_\p-\la} \otimes_E V$, and satisfies $\ep \tilde f = \tilde f$. Since 
\[ f(g k)(1) = (k^{-1} \circ f)(g)(1) = \rho(\Psi_M(k^{-1})) f(g)(\Psi_N(k^{-1}))\] 
for all $k \in G_0$, we have in particular $\tilde f(g \overline{p}) = \rho(\Psi_M(\overline{p})^{-1}) \circ \tilde f(g)$ if $\overline{p} \in \overline{P}_0 = \overline{N}_0 M_0$. Thus $\overline{P}_0$ acts trivially on $\tilde f$ when we endow $\widetilde H^0_{F_\p-\la} \otimes_E V$ with the diagonal action of $\overline{P}_0$.

Conversely, let $h$ be an element of $\left( \widetilde H^0_{F_\p-\la} \otimes_E V \right)^{\overline{P}_0}$. We regard $h$ as a function $G(\A) \to V$ such that $h(g\overline{p}) = \rho(\overline{p})^{-1} h(g)$ for all $\overline{p} \in \overline{P}_0$, and define $f$ to be the function $G(\A) \times N_0 \to V$ given by $f(g)(n) = h(gn^{-1})$. Then clearly $f(g)(-) \in C^\la(N_0, E)$ for each $g \in G(\A)$. Thus we can regard $f$ as a function $G(\A) \to C^\la(N_0, E)$. We calculate that
\begin{align*}
f(g u)(n) &= h(g u n^{-1})\\
&= h(g (nu^{-1})^{-1})\\
&= h\left(g \left[ \Psi_{\overline{N}}(nu^{-1}) \Psi_{M}(nu^{-1}) \Psi_{N}(nu^{-1})\right]^{-1}\right)\\
&= h\left(g \Psi_{N}(nu^{-1})^{-1} \Psi_{M}(nu^{-1})^{-1} \Psi_{\overline{N}}(nu^{-1})^{-1} \right)\\
&= h\left(g \Psi_{N}(nu^{-1})^{-1} \Psi_{M}(nu^{-1})^{-1} \right) \tag{by $\overline{N}_0$-invariance}\\
&= \rho(\Psi_M(nu^{-1})) h(g \Psi_{N}(nu^{-1})^{-1})\tag{since $M_0$ acts via $\rho^{-1}$} \\
&= (u \circ f(g))(n).
\end{align*}
Hence $f$ is an element of $\mathcal{L}(\Ind_{\overline{P}}^{G_0} V)$, and evidently we have $\tilde f = h$.

Finally we note that the map $f \mapsto \tilde f$ is clearly $G(\A^{(\p)})$-equivariant, and hence $\ep \HH^{(\p)}(G) \ep$-equivariant.
\end{proof}

We now consider the action of Hecke operators at $\p$. On the left-hand side, we have an action of $\Sigma^+$, via the isomorphism $E[\Sigma^+] = e_{G_0} \HH_\p^+(G) e_{G_0}$. On the right, let $M^+$ be the monoid $\{ z \in M(L) \mid z \overline{N}_0 z^{-1} \subseteq \overline{N}_0\}$; then for any $W \in \Rep_{\rm la, c}(\overline{P})$, there is an action of $M^+$ (and hence of $\Sigma^+$) on $W^{\overline{N}_0}$ constructed in \cite[\S 3.4]{emerton-jacquet}. We extend this to an action of $M_0 \Sigma^+$ on the tensor product by letting $\Sigma^+$ act trivially on the second factor.

\begin{proposition}
The homomorphism of the previous proposition is $\Sigma^+$-equivariant.
\end{proposition}

\begin{proof}
 The action of $z \in \Sigma^+$ on the right-hand side is given by 
\[(z \circ f)(g) = \frac{1}{[\overline{N}_0 : z \overline{N}_0 z^{-1}]} \sum_{n \in \overline{N}_0 / z \overline{N}_0 z^{-1}} f(g z n).\]
On the other hand, the action on the right-hand side obtained by transporting the action on the left via the isomorphism of the previous proposition is given by the similar but apparently more complicated formula
\[ (z \circ f)(g) = \frac{1}{[G_0 : G_0 \cap z G_0 z^{-1}]} \sum_{\gamma \in G_0 / G_0 \cap z G_0 z^{-1}} \rho(\Psi_M(z\gamma)) f\left(g z \gamma \Psi_N(z\gamma)^{-1}\right).\]
However, from the Iwahori factorisation of $G_0$ we have $z G_0 z^{-1} = (z \overline{N}_0 z^{-1})(z M_0 z^{-1})(z N_0 z^{-1})$. As $z \in Z(M)$ the middle term is $M_0$; and as $z \in \Sigma^+$, $z N_0 z^{-1} \supseteq N_0$. Thus $G_0 \cap z G_0 z^{-1} = (z \overline{N}_0 z^{-1}) M_0 N_0$, and hence any set of coset representatives for $\overline{N}_0 / z \overline{N}_0 z^{-1}$ is also a set of coset representatives for $G_0 / G_0 \cap z G_0 z^{-1}$. Since for all $\gamma \in \overline{N}_0$ we have $\Psi_N(z\gamma) = 1$ and $\rho(\Psi_M(z\gamma)) = \rho(z) = \mathrm{id}_V$, the two actions coincide.
\end{proof}

Now let $\lambda$ be a character $\Sigma \to E^\times$. By proposition \ref{lemma:atkin-lehner-algebra}, this is equivalent to a finite slope $E$-valued system of eigenvalues for $\mathcal{A}^+(G)$. 

\begin{proposition}
 As representations of $\HH^{(\p)}(G)$, we have 
\[M(e, \bbone, V, k_0)^{\mathcal{A}^+(G) = \lambda} = \ep \left[J_{\overline{P}}\left( \widetilde H^0_{F_\p-\la}\right) \otimes_E V\right]^{M_0, \Sigma = \lambda}\]
for all $k_0 \ge k(V)$.
\end{proposition}

\begin{proof} First we note that for all $k_0 \ge k(V)$, the map $M(e, \bbone, V, k_0) \to \varprojlim_k M(e, \bbone, V, k)$ induces an isomorphism on the $\lambda$-eigenspaces, by corollary \ref{corr:automorphic-maximally-overconvergent}. So, by the previous proposition, the left-hand side is just $\ep \left(\widetilde H^0_{F_\p-\la} \otimes_E V\right)^{\overline{P}_0, \Sigma^+ = \lambda}$, which we can also write as $\ep \left[\left(\widetilde H^0_{F_\p-\la}\right)^{\overline{N}_0} \otimes_E V\right]^{M_0, \Sigma^+ = \lambda}$.

We deduce from proposition 3.2.12 of \cite{emerton-jacquet} that this coincides with
$\ep \left[\left(\widetilde H^0_{F_\p-\la}\right)^{\overline{N}_0} \otimes_E V\right]_{\rm fs}^{M_0, \Sigma^+ = \lambda}$, where $(-)_{\rm fs}$ denotes Emerton's finite slope part functor. By {\it op.cit.},~3.2.9 this is simply $\ep \left[\left(\widetilde H^0_{F_\p-\la}\right)^{\overline{N}_0}_{\rm fs} \otimes_E V\right]^{M_0, \Sigma^+ = \lambda}$; since the Jacquet functor $J_{\overline{P}}$ is defined as $\left(-\right)^{\overline{N}_0}_{\rm fs}$, and all of these isomorphisms clearly commute with the action of $G(\A^{(\p)})$, the result follows.
\end{proof}

\subsection{The eigenvariety machine}

We shall now use the above results to construct a geometric object parametrising overconvergent eigenforms. This object is analogous to the Coleman-Mazur eigencurve for overconvergent modular forms, and the methods we use to construct it are based on the Hecke algebra construction of the eigencurve given in \cite{CMeigen}, as subsequently extended in \cite{buzzard-eigen} and \cite{bellaiche-chenevier}. In this section, we will regard all rigid spaces as being defined over $E$ via base extension, and for a rigid space $X$ we will write $\OO(X)$ for the algebra of analytic functions on $X$, rather than $C^\an(X, E)$.

\begin{definition}
Let $X$ be an affinoid rigid space over $E$, $M$ a Banach module over $\OO(X)$ satisfying (Pr), and $\mathbf{T}$ a commutative $E$-algebra endowed with a map $\mathbf{T} \to \End_{\OO(X)}(M)$; and suppose that $\phi \in \mathbf{T}$ is a fixed element which acts as a compact endomorphism of $M$.

For $E'$ a discretely valued extension of $E$, we say an $E$-algebra homomorphism $\lambda : \mathbf{T} \to E'$ is an $E'$-valued system of eigenvalues (for $M$) if there is a point $P \in X(E')$ (giving a homomorphism $\OO(X) \to E'$) and a nonzero vector $m \in M \otimes_{\OO(X)} E'$ such that $\alpha x = \lambda(\alpha) x$.

We say a system of eigenvalues is $\phi$-finite if $\lambda(\phi) \ne 0$.
\end{definition}

\begin{theorem}
To the data $(X, M, \mathbf{T}, \phi)$ as above, we can associate a reduced $E$-rigid space $D_\phi$, endowed with a $E$-algebra homomorphism $\Psi: \mathbf{T} \to \OO(D_{\phi})$ and a map of $E$-rigid spaces $f: D_{\phi} \to X$, satisfying the following conditions:
\begin{enumerate}
 \item The map $\nu: (f, \Psi(\phi)^{-1})$ is a finite map $D_\phi \to X \times \mathbb{G}_m$.
 \item The map $\mathbf{T} \otimes \OO_{f(x)} \to \OO_x$ induced by $\Psi$ is surjective for every point $x$ of $D_\phi$, where $\OO_x$ and $\OO_{f(x)}$ denote the local rings of $D_\phi$ and $X$ at the points $x$ and $f(x)$.
 \item For each discretely valued extension $E' / E$, the map associating to a point $x \in D_\phi(E')$ the homomorphism $\mathbf{T} \to E'$ given by composing $\Psi$ with evaluation at $x$ gives a bijection between $D_\phi(E')$ and the set of $\phi$-finite $E'$-valued systems of eigenvalues for $M$. 
\end{enumerate}
Furthermore, $D_\phi$ is uniquely determined by these conditions.
\end{theorem}

\begin{proof}
For the existence statement, see \cite[\S 5]{buzzard-eigen}. Condition (2) above does not appear there in quite the same form, but Buzzard's construction of $D_\phi$ shows that each point has an affinoid neighbourhood of the form $\Max A$ where $A$ is defined as the image of $\mathbf{T}$ in a certain endomorphism algebra, from which the statement (2) follows. 

The uniqueness is \cite[\S 7.2.3]{bellaiche-chenevier}, where it is shown more generally that an eigenvariety is determined up to isomorphism by conditions (i), (ii), and the systems of eigenvalues corresponding to a sufficiently dense subset of points of $D_\phi(\overline{\Q}_p)$.
\end{proof}

\begin{remark}
There is a functor from the slice category of affinoid spaces with a map to $X$ to the category of sets, which maps a space $Y \xrightarrow{\pi} X$ to set of $\phi$-finite systems of eigenvalues for the pullback $\pi^*(M) = M \widehat\otimes_{\OO(X)} \OO(Y)$. It is clear that $D_\phi$ represents this functor. However, in general the pullback of $D_\phi$ does not satisfy condition (2) above, and hence $\pi^* D_\phi$ is is not the eigenvariety for $(Y, \pi^* M, \mathbf{T}, \phi)$ -- the formation of $D_\phi$ is not functorial with respect to arbitrary base change (although this is true for flat base change). Nevertheless the pullback of $D_\phi$ and the eigenvariety of the pullback have the same nilreduction.
\end{remark}

\begin{corollary}\label{corr:finiteslope2}
 If $\phi_1$ and $\phi_2$ are two elements of $\mathbf{T}$ which both act as compact endomorphisms of $M$, and every $\phi_1$-finite system of eigenvalues for $M$ is also $\phi_2$-finite and vice versa, then the eigenvarieties associated to $(X, M, \mathbf{T}, \phi_1)$ and $(X, M, \mathbf{T}, \phi_2)$ are isomorphic.
\end{corollary}

\begin{proof} For $i = 1, 2$ let $D_i$, $f_i$, and $\Psi_i$ denote the eigenvarieties and their structure maps corresponding to $\phi_1$ and $\phi_2$. The function $\Psi_1(\phi_2)$ on $D_1$ is analytic and has no zeros by hypothesis, so $\nu = (f_1, \Psi_1(\phi_2)^{-1}) : D_1 \to X \times \mathbb{G}_m$ is well-defined. The image of $\nu$ is evidently contained in the spectral variety $Z_2 \subseteq X \times \mathbb{G}_m$, which is the zero locus of the characteristic power series of $\phi_2$.

I claim that $\nu$ is a finite map $D_1 \to Z_2$. It will follow from this that $D_1$ satisfies the properties that uniquely determine the eigenvariety associated to $(X, M, \mathbf{T}, \phi_2)$, so $D_1 = D_2$.

We know from \cite[thm.~4.6]{buzzard-eigen} that $Z_2$ is admissibly covered by affinoid subdomains $Y$ such that the projection $\pi: Z_2 \to X$ restricts to a finite flat map $Y \to W =  \pi(Y)$, and $Y$ is disconnected from its complement in $\pi^{-1}(W)$. Let $Y$ be any such affinoid, and let $M_Y$ be the finite-rank $\mathbf{T}$-stable $\OO(W)$-module direct summand of $M \widehat\otimes_{\OO(X)}\OO(W)$ attached to $Y$ by \cite[thm.~3.3]{buzzard-eigen}. Then it is immediate from the uniqueness result above that $\nu^{-1}(Y)$ is the eigenvariety attached to the data $(Z_Y, M_Y, \mathbf{T}, \phi_1)$. As $M_Y$ has finite rank over $\OO(Z_Y)$, this is evidently finite over $Z_Y$, and in particular it is finite over $Y$.
\end{proof}

In applications of this result the base space will typically not be affinoid, and the modules involved will be overconvergent forms whose radius of convergence will be forced to tend to zero as we approach the boundary, so we need the following extension:

\begin{definition}[{\cite[\S 5]{buzzard-eigen}}]\label{def:links}
 Let $M'$, $M$ be Banach $\OO(X)$-modules with a continuous action of $\mathbf{T}$ and compact $\phi$-action as above. A map $\alpha: M' \to M$ is a {\bf primitive link} if it is a continuous, $\mathbf{T}$-equivariant, $\OO(X)$-linear map, and there exists another such map $c : M \to M'$ which is also compact, such that $\alpha \circ c = \phi_{M}$ and $c \circ \alpha = \phi_{M'}$.
A {\bf link} is a map which is the composite of a finite sequence of primitive links. 
\end{definition}

Lemma 5.6 of \cite{buzzard-eigen} shows that if $M$ and $M'$ are related by a link $\alpha$, then $\alpha$ induces an isomorphism between the eigenvarieties of $(X, M, \mathbf{T}, \phi)$ and $(X, M', \mathbf{T}, \phi)$.

Finally, let us suppose that we are given the following data: a rigid space $\mathcal{W}$; for each open affinoid subdomain $X \subset \mathcal{W}$ an $\OO(X)$-module $M(X)$ with a $\mathbf{T}$-action such that $\phi$ is compact; and for each inclusion of affinoids $Y \subset X \subset \mathcal{W}$ a link $\alpha_{XY}: M(Y) \to M(X) \widehat\otimes_{\OO(X)} \OO(Y)$, satisfying the obvious compatibility condition $\alpha_{XZ} = \alpha_{YZ} \alpha_{XY}$ whenever $(X,Y,Z)$ are affinoids with $Z \subset Y \subset X \subset \mathcal{W}$. 

\begin{theorem}\label{thm:link-eigenmachine}
 To the above data we may associate a rigid space $D_\phi$ over $\mathcal{W}$ such that for any open affinoid $X \subset \mathcal{W}$, the pullback of $D_\phi$ to $X$ is canonically isomorphic to the eigenvariety attached to $(X, M(X), \mathbf{T}, \phi)$. Furthermore, if $\mathcal{W}$ is equidimensional of dimension $n$, $D_\phi$ is also equidimensional of dimension $n$; and the image of any component of $D_\phi$ in $\mathcal{W}$ is a Zariski-open subset of $\mathcal{W}$.
\end{theorem}

\begin{proof}
 The existence is \cite[Construction 5.7]{buzzard-eigen}. The equidimensionality statement is \cite[prop.~6.4.2]{chenevier-families}.
\end{proof}

\subsection{Eigenvarieties for automorphic forms}

We now apply the machinery of the preceding subsection to our spaces of overconvergent forms. Corollary \ref{corr:wild-level-lowering} above shows that all of the Hecke eigenvalues which we are interested in can be seen in the spaces $M(\ep \otimes e_{G_0}, X, V, k)$ for idempotents $\ep \in \HH^{(\p)}(G)$, so we shall apply the ``eigenvariety machine'' to these spaces. Let us fix $V$ and $\ep$ and define $M(X, k) = M(\ep \otimes e_{G_0}, X, V, k)$; we also write $M(X)$ for $M(X,k)$ when $k = \max(k(V), k(X))$.

Let $S$ be a finite set of places including $\p$ and all primes at which our tame idempotent $\ep$ is ramified, and define 
\[\mathbf{T} = \mathcal{A}^+(G) \otimes_E \sideset{}{^\prime}\bigotimes_{\q \notin S} \ep \HH(G(F_\q)) \ep.\]
This is a commutative, unital $E$-algebra. 

\begin{remark}
 One could also choose commutative subalgebras of the algebras $\ep \HH(G(F_\q)) \ep$ at primes $\q \in S \setminus \{\p\}$, and include these in the tensor product; the only requirement is that $\mathbf{T}$ be commutative. 
\end{remark}

Now let us fix an element $\eta \in \Sigma^{++}$, and let $\phi$ be the corresponding element of $\mathcal{A}^{+}(G)$. We know from lemma \ref{lemma:finiteslope1} and corollary \ref{corr:finiteslope2} that the resulting space will not depend on this choice. The final check we must make is the existence of the appropriate link maps:

\begin{lemma} If $Y \subseteq X$ are affinoid subdomains of $\widehat H_0^{arith}$ defined over $E$, then there is a $\mathbf{T}$-equivariant $\OO(Y)$-module homomorphism
\[ \alpha_{XY}: M(Y)  \to M(X) \widehat\otimes_{\OO(X)} \OO(Y) \]
which is a link; and these links satisfy the compatibility condition $\alpha_{YZ} \alpha_{XY} = \alpha_{XZ}$ if $X \supseteq Y \supseteq Z$.
\end{lemma}

\begin{proof} 
Since $Y \subseteq X$ we have $k(Y) \le k(X)$. Note that $M(X) \widehat\otimes_{\OO(X)} \OO(Y) = M(Y, k(X))$. If $k(Y) = k(X)$ there is nothing to prove (the identity map is a link). Suppose $k(X) = k(Y) + 1$; then we take $\alpha$ to be the inclusion map $M(Y, k(Y)) \into M(Y, k(X))$, and part \eqref{item:automorphic-compactness} of theorem \ref{thm:spacesM} shows that there is a map $c$ satisfying the appropriate requirements. 

If $k(X) > k(Y)+1$ then one clearly obtains a chain of such maps whose composition is the required (non-primitive) link. The compatibility condition is also clear, since the composition of the inclusion maps $M(Z,k(Z)) \into M(Z, k(Y)) \into M(Z, k(X))$ clearly coincides with the inclusion $M(Z, k(Z)) \into M(Z, k(X))$.
\end{proof}

Applying theorem \ref{thm:link-eigenmachine}, we deduce the following result:

\begin{theorem}
There exists an eigenvariety of tame type $\ep$, namely a reduced $E$-rigid space $D(\ep, V)$ endowed with a $E$-algebra homomorphism $\Psi: \mathbf{T} \to \OO(D(\ep, V))$ and a map of $E$-rigid spaces $f: D(\ep, V) \to \widehat H_0^{arith}$, whose points over any discretely valued extension $E'/E$ parametrise finite slope $E'$-valued systems of eigenvalues arising in the spaces $M(\ep \otimes e_{G_0}, X, V, k)$ for affinoids $X \subseteq \widehat H_0^{arith}$. 

The space $D(\ep, V)$ is equidimensional of the same dimension as $\widehat H_0^{arith}$, and its image in $\widehat H_0^{arith}$ is contained in a finite union of components of the latter, in each of which it is Zariski-open.
\end{theorem}

\subsection{Density of classical points}

In this section, we'll see that in certain cases theorem \ref{thm:classicality} can be used to show that there are ``many'' points in $D(\ep, V)$ arising from classical automorphic forms. We begin by recalling the useful definition of {\it accumulation} of a subset of a rigid space:

\begin{definition}[{\cite[\S 3.3.1]{bellaiche-chenevier}}]
Let $X$ be a rigid space over a complete nonarchimedean valued field $K$, and $|X|$ its underlying set of points. If $x \in X$ and $S \subset |X|$, we say $S$ {\it accumulates at $x$} if there is a basis of affinoid neighbourhoods $Y$ of $x$ such that $S \cap |Y|$ is Zariski dense in $Y$.
\end{definition}

Our next definition is slightly complicated to state. An {\it open affine cone} in $\R^d$ is a subset of the form $C = \{x \in \R^d: \phi_1(x) > c_1, \dots, \phi_n(x) > c_n\}$ for linear functionals $\phi_1, \dots, \phi_n$ and non-negative $c_i \in \R$. Recall that $T = Z(M) \cap S$, a torus over $F_\p$. Let $\kappa \in \widehat H_0^{arith}$, and let $C$ be an open affine cone in $X^\bullet(T) \otimes \R$; we say that $\kappa$ {\it belongs to $C$} if it is locally algebraic and the restriction of its algebraic part to $T$ lies in $C$.

\begin{definition}
Let $X$ be a component of $\widehat H_0^{arith}$. We say $X$ is {\it good} if for any nonempty open affine cone $C$ in $X^\bullet(T) \otimes \R$, the points of $X$ belonging to $C$ are Zariski-dense in $X$, and accumulate at every locally algebraic point in $X$.
\end{definition}

\begin{corollary}
Suppose $V$ is locally algebraic, and $X$ is a good component of $\widehat H_0^{arith}$. Then points corresponding to classical automorphic forms are Zariski-dense in the preimage of $X$ in $D(\ep, V)$.
\end{corollary}

\begin{proof}
We follow the argument of \cite[\S 6.4.5]{chenevier-families}. Let $\eta \in \Sigma^{++}$, and let $Z_\eta(\ep, V)$ be the spectral curve associated to $\eta$. We let $D(\ep, V)_X$ and $Z_\eta(\ep, V)_X$ be the preimages of $X$ in $D(\ep, V)$ and $Z_\eta(\ep, V)$; we then have a commutative diagram
\[\xymatrix{
  D(\ep, V)_{X} \ar[r]^\mu \ar[rd]^{f_D} & Z_\eta(\ep, V)_{X} \ar[d]^{f_Z} \\
 & X
  }
\]
where $\mu$ is a finite map and $f_D$ and $f_Z$ are the projections to weight space.

If $Z_\eta(\ep, V)_X$ is empty, there is nothing to prove. Otherwise, let $Z$ be a component of $Z_\eta(\ep, V)_X$. We know that the image of $Z$ in $X$ is Zariski-open. Since $X$ is good, locally algebraic points are Zariski dense in $X$; so there is a locally algebraic point $\kappa_0$ in the image of $Z$. Let $z_0$ be a point of $Z$ above $\kappa_0$. By construction, $Z$ is admissibly covered by affinoids $U$ such that $U' = f_Z(U)$ is an affinoid in $X$, $f_Z: U \to U'$ is finite and flat, and $U$ is a connected component of $f_Z^{-1}(U')$; let $U$ be an affinoid neighbourhood of $z_0$ of this type.

Since $U$ is quasicompact and $\Psi(\eta)$ does not vanish on $U$, the function $\ord_p \Psi(\eta)$ is bounded above on $U$. Let $M$ be its supremum. Then there is a nonempty open affine cone $C \subseteq X^\bullet(T)$ such that if $\lambda \in C$, $M$ is a small slope for $\eta$ relative to $\lambda$. Since $X$ is a good component, locally algebraic characters whose algebraic part lies in $C$ accumulate at $\kappa_0$, and hence are Zariski dense in some neighbourhood $U''$ of $\kappa_0$ contained in $U'$. We define $Y = f_D^{-1}(U'') \cap \mu^{-1}(U)$, which is an affinoid neighbourhood of $\mu^{-1}(z_0)$ in $D(\ep, V)$. By \cite[lemma 6.2.8]{chenevier-families}, we deduce that the set of points of $Y$ whose image in $U''$ is locally algebraic is Zariski dense in $Y$. But any such point must be classical, by theorem \ref{thm:classicality}; so classical points are Zariski dense in $Y$, and hence in $\mu^{-1}(Z)$. Since $Z$ was arbitrary, the result follows.
\end{proof}

\begin{remark}
 If $F_\p = \Q_p$ and $Z(G)(\OO_F)$ is trivial, then it is easy to see that every geometric connected component of $\widehat H_0 = \widehat H_0^{arith}$ is good, since the algebraic weights evidently accumulate at the origin in $\widehat H_0$. This can be extended to general $F_\p$ using the explicit description of $\widehat \OO_{F,\p}$ given in \cite{ST-fourier}. However, for $G = \Res_{K/\Q} \mathbb{G}_m$ where $K = \mathbb{Q}(\sqrt[3]{2})$, and any prime $p$, locally algebraic weights are not Zariski dense in $\widehat G_0^{arith}$ (where $G_0 = G(\Z_p)$), and hence the Zariski closure of the classical weights has dimension strictly smaller than the whole eigenvariety.
\end{remark}

\section{Examples}
\label{chap:examples}

We shall now make the above machinery rather more explicit, and show that for certain specific examples of groups $G$ satisfying the hypotheses, one recovers results already in the literature (or generalisations thereof).

\subsection{Tori}

If $G$ is a torus over $F$, then the only possible parabolic subgroup of $G$ is $G$ itself. We may take $G_0$ to be any compact open subgroup of $G(F_\p)$, and $G_1$ an arbitrary good $F_\p$-analytic subgroup of $G_0$. The unipotent radical $N$ is the trivial group, so $N_k =\{1\}$ for all $k$. Thus for any affinoid $X \subset \widehat G_0^{arith}$ and any locally algebraic $V$, we have $\mathcal{C}(X, V, k) = C^\an(X, E)\ \widehat\otimes_E\ V$, with the action of $G_0$ on $V$ being the natural one and the action on $C^\an(X, E)$ via $\Delta_X$.

Let $\ep$ be an idempotent of the form $e_U$, for a compact open subgroup $U \subseteq G(\Af)$. Then for any $k \ge k(X)$, $M(\ep \otimes e_{G_0}, X, \bbone, k)$ is precisely the set of functions $f: G(F) \backslash G(\A) \to C^\an(X, E)^\times$ satisfying $f(g u) = \Delta_X(u_\p)^{-1} f(g)$ for $u \in G_0 U G_\infty^\circ$. This differs from the space defined on \cite[p7]{buzzard-families} (for the special case of the torus over $\Q$ given by $\Res_{K/\Q} \mathbb{G}_m$ for $K$ a number field) only in that we have $\Delta_X(u_\p)^{-1}$ where Buzzard has $\Delta_X(u_\p)$. 

\subsection{Unitary groups}

Let us suppose we are in the situation of \cite{chenevier-families}: $G$ is a unitary group over $\Q$ which is split at $p$, so $G$ is isomorphic to $\GL_n$ as an algebraic group over $\Q_p$. In this case, we take $S$ to be the diagonal torus and $P_{\rm min}$ to be the usual Borel subgroup of upper-triangular matrices. The parabolic subgroups containing $P_{\rm min}$ are groups of block upper-triangular matrices, each determined by the set $J$ of its jumps, which is contained in $\{1, \dots, n-1\}$; the free simple roots are $L_i - L_{i+1}$ for $i$ in $J$, where $L_i$ sends a diagonal torus element to its $(i,i)$ entry. It is easy to check directly in this case that the big Bruhat cell $C$ is an open $\Z_p$-subscheme of $G$, defined by the non-vanishing of the functions $(Z_{i1})_{i \in J}$, where $Z_{i1}$ is the function mapping a matrix to the determinant of its top left-hand $i \times i$ submatrix. 

Let us define $G_0$ to be the parahoric subgroup of $\GL_n(\Z_p)$ corresponding to $P$, containing those matrices whose reduction modulo $p$ lies in $P$; and let $G_1$ be the subgroup of $G_0$ whose reduction mod $p$ lies in $N$. It is easy to check that $G_1$ is normal in $G_0$ and that $G_1$ is decomposable in the sense of definition \ref{def:decomposable}, so $G_0$ is an admissible subgroup; and both $N_0$ and $N_1$ are equal to $N(\Z_p)$.

In the case where $J = \{1, \dots, n-1\}$, so $P = P_{\rm min}$ and $S = M = H$, we see that the affinoid subvariety $\mathscr{F}$ of \cite[\S 3.2]{chenevier-families} is exactly the image of $\mathcal{N}_0$ in the flag variety $\overline{P} \backslash G$. Hence the space $A(\mathscr{F})$ of \cite[\S 3.2]{chenevier-families} is our $C^\an(\mathcal{N}_0, \Q_p)$; and the representations $\mathscr{S}_r$ of \cite[\S 3.6]{chenevier-families}, defined over certain affinoids $\mathscr{W}_r \subseteq \widehat H_0$ of characters analytic on the principal congruence subgroup $M_1$, are easily seen to coincide with our spaces $\mathcal{C}(\mathscr{W}_r, \bbone, 1)$. 

The more general construction given in \cite[\S 7.3]{bellaiche-chenevier} uses slightly different conventions, with actions defined by left multiplication rather than right multiplication; in our optic this corresponds to choosing the parabolic subgroup $P$ to be the {\it lower} triangular Borel, while still choosing $G_0$ to be the upper triangular Iwahori subgroup. Then one can check that the space $\mathcal{C}(X, k)$ constructed in \cite{bellaiche-chenevier}, for $X$ either an open affinoid or a closed point in $\widehat H_0$ and $k \ge k(X)$, is isomorphic as a $G_0$-representation to what we would denote by $\mathcal{C}(X^{-1}, \bbone, k)$, where $X^{-1}$ denotes the image of $X$ under the involution of $\widehat H_0$ that sends a character to its inverse, and the map giving the isomorphism is given by $f(x) \mapsto f(x^{-1})$.

\subsection{$\GL_2$ and restriction of scalars}

In this section we'll carry out a purely local calculation relating various representations of the group $G = \Res_{K/\Q_p} \GL_2$, for $K$ a finite extension of $\Q_p$.

The $\Q_p$-split part of the diagonal torus in $G$ is a maximal split torus, of rank 2; with respect to this torus, there is only one positive and one negative root (of multiplicity $d = [K : \Q_p]$). We choose the signs so that $P_{\rm min}$ is the standard upper-triangular Borel subgroup; then $M$ is the rank $2d$ diagonal torus. The natural choice for $G_0$ is the upper triangular Iwahori subgroup of elements congruent to $\begin{pmatrix} * & * \\ 0 & * \end{pmatrix} \bmod \pi_K$. If we take $G_1$ to be the kernel of reduction modulo $\pi_K^a$, for $a \in \Z_{\ge 1}$, then $G_1$ is normal in $G_0$, and for sufficiently large $a$ it is a good $\Q_p$-analytic open subgroup. Thus $G_0$ is admissible in the sense of definition \ref{def:admissible}. 

In this case, $N_0 = \OO_K$. Hence $N_0$ is naturally the $\Q_p$-points of the affinoid rigid space $\mathcal{N}_0 = \Res_{K / \Q_p}\left(\mathbf{B}(1) / K\right)$, where $\mathbf{B}(1)/K$ is the standard unit ball over $K$; see \cite[\S 2.3]{emerton-memoir} for a discussion of restriction of scalars in the rigid analytic context. Concretely, this is isomorphic to $\mathbf{B}(1)^d / \Q_p$ with coordinates $Z_1, \dots, Z_d$ corresponding to a basis $z_1, \dots, z_d$ of $\OO_K$ as a $\Z_p$-module. For $\tau \in \OO_K$ the evaluation map is given by $f(\tau) = f(\zeta_1(\tau), \dots, \zeta_d(\tau))$ where $\zeta_i$ is the coordinate functional corresponding to $z_i$. Similarly, the affinoid subdomain $\mathcal{N}_k$ is just the affinoid ball attached to the sublattice $p^k \pi_K^a \OO_K$, which is $\Res_{K / \Q_p}\left(\mathbf{B}(|p^k \pi_K^a|) / K\right)$. 

More generally, for each $r \in p^\Q$ with $r \le 1$ we may define the open affinoid $\mathcal{N}(r) = \Res_{K / \Q_p} (\mathbf{B}(r) / K)$, which is an open affinoid subdomain. Let us set $\OO_K(r) = \{ x \in \OO_K : |x| \le r\}$. Then we can define $\mathcal{D}(r)$ to be the union of the translates of $\mathcal{N}(r)$ by the elements of $\OO_K / \OO_K(r)$, extending the definition of $\mathcal{D}_k$.

We wish to compare this with the slightly different theory developed in \cite[\S 8]{buzzard-eigen} (and \cite[\S 1.2]{yamagami}, which is very similar). Let $I$ be the set of embeddings $K \into E$; we assume that $E$ is large enough assume that $|I| = d$. Then $z \mapsto (\sigma(z))_{z \in I}$ gives an embedding of $\OO_K$ into $\mathbf{B}(1)^d$. Let $\mathcal{X}_0 = \mathbf{B}(1)^d$, and let $\mathcal{X}(r)$, for $r \le 1 \in p^\Q$, be the union of the polydiscs $\mathbf{B}(r)^d$ centred at the points of $\OO_K / \OO_K(r)$ in $\mathcal{X}_0$. Then $\mathcal{X}(r)$ is an open affinoid subdomain of $\mathcal{X}_0$ over $\Q_p$. (This is the space denoted by $B_r$ in \cite[\S 8]{buzzard-eigen}).

\begin{definition}
We shall call $\mathcal{D}(r)$ and $\mathcal{X}(r)$ respectively the type 1 and type 2 thickenings of $\OO_K$ (of radius $r$).
\end{definition}

The $\Q_p$-points of $\mathcal{X}(r)$ and $\mathcal{D}(r)$ are canonically $N_0 = \OO_K$, for every $r$, and the right action of $G_0$ on $N_0$, which is given by $\begin{pmatrix} a & b \\ c & d \end{pmatrix} \circ z = \frac{b + dz}{a + cz}$, extends to an action of $G_0$ on $\mathcal{D}(r)$ and on $\mathcal{X}(r)$ via rigid-analytic automorphisms. (This is clear for $\mathcal{D}(r)$, and for $\mathcal{X}(r)$ it is \cite[lemma 8.1(b)]{buzzard-eigen}).

\begin{proposition}
Let $\mu = \sqrt{|\Delta|}$, where $\Delta$ is the discriminant of the field extension $K/\Q_p$. Then there exist maps of rigid spaces $\mathcal{D}(r) \to \mathcal{X}(r)$ and $\mathcal{X}(r\mu) \to \mathcal{D}(r)$, commuting with the action of $G_0$, which are the identity on $\OO_K$.
\end{proposition}

\begin{proof}
We first consider the case $r = 1$. If we identify $\mathcal{D}(1)$ and $\mathcal{X}(1)$ with $\mathbf{B}(1)^d$, via the coordinates $Z_i$ on $\mathcal{D}(1)$ introduced above and the natural coordinates $X_\sigma$ on $\mathcal{X}$, then the linear map sending $(Z_1, \dots, Z_d)$ to $(\sum_{j=1}^d \sigma(z_j) Z_j)_{\sigma \in I}$ commutes with the embeddings of $\OO_K$ into the two factors. Since $\OO_K$ is Zariski dense in both $\mathcal{X}(1)$ and $\mathcal{D}(1)$, this map must commute with the $G_0$-action. 

Similarly, for any $r$ we have an embedding of the identity component of $\mathcal{D}(r)$ into the identity component of $\mathcal{X}(r)$, which commutes with the embedding of $\left\{x \in \OO_K\ \middle|\ |x| \le p^r\right\}$ into both, and we can extend this to the other components by translating by elements of $\OO_K$.

To define the second map, we note that the matrix $A$ whose $i,j$ term is $\sigma_i(z_j)$, for any ordering $\sigma_1, \dots, \sigma_d$ of $I$, satisfies $(\det A)^2 = \Delta_{K/\Q_p}$. Consequently, $A^{-1}$ gives a map from the disc $B(r \mu)$ to $B(r)$. Identifying these with the identity components of $\mathcal{X}(r\mu)$ and $\mathcal{D}(r)$, the resulting map is clearly the identity on $\OO_K(r\mu)$. Extending to the remaining components by translation by $\OO_K$, we obtain the required map; and as before, it is $G_0$-equivariant by the Zariski-density of $\OO_K$.
\end{proof}

We thus obtain the following result relating the representations obtained from functions on $\mathcal{D}(r)$ and $\mathcal{X}(r)$:

\begin{corollary}\label{corr:thickenings}
Let $\chi_1, \chi_2$ be locally $\Q_p$-analytic characters $\OO_K \to E^\times$, and $r \in p^\Q$ sufficiently small that $\chi_1$ is analytic on $1 + \OO_K(r)$. Let us endow $C^{\an}(\mathcal{D}(r), E)$ and $C^{\an}(\mathcal{X}(r), E)$ with the action of $G_0$ given by $\begin{pmatrix} a & b \\ c & d \end{pmatrix} \circ f = \chi_1(a + cz) \chi_2(ad - bc)f\left(\frac{b + dz}{a + cz}\right)$. 
\begin{enumerate}
\item There are continuous injective $G_0$-equivariant maps $C^{\an}(\mathcal{X}(r), E) \to C^{\an}(\mathcal{D}(r), E)$ and $C^{\an}(\mathcal{D}(r), E) \to C^{\an}(\mathcal{X}(r\mu), E)$, with dense image.
\item The compositions $C^{\an}(\mathcal{X}(r), E) \to C^{\an}(\mathcal{X}(r\mu), E)$ and $C^{\an}(\mathcal{D}(r), E) \to C^{\an}(\mathcal{D}(r\mu), E)$ are the natural embeddings.
\item The map $\alpha: C^{\an}(\mathcal{X}(r), E) \into C^{\an}(\mathcal{D}(r), E)$ is a primitive link, in the sense of definition \ref{def:links}, for the endomorphisms of the two spaces given by $f(z) \mapsto f(\lambda z)$ for any $\lambda \in \Z_p$ with $|\lambda| < \mu$.
\end{enumerate}
\end{corollary}

\begin{proof}
The first two statements are clear from the previous proposition. For part (iii), in the coordinates introduced above, the map $\alpha$ is given by the matrix $A$ with $|\det A| = \mu$; so the map $c: C^{\an}(\mathcal{D}(r), E) \into C^{\an}(\mathcal{X}(r), E)$ given by the matrix $\lambda A^{-1}$ is well-defined and compact, and it is clear that the two compositions are the maps on $C^{\an}(\mathcal{D}(r), E)$ and $C^{\an}(\mathcal{X}(r), E)$ given by multiplication by $\lambda$.
\end{proof}

\subsection{Quaternion algebras over a totally real field}

As a final example, we consider the case of quaternionic Hilbert modular forms studied in \cite{yamagami} and \cite[part III]{buzzard-eigen}. Let $F$ be a totally real field and $B$ a totally definite quaternion algebra over $F$ (i.e.~$B$ is ramified at all infinite places of $F$). $B^\times$ is naturally the $F$-points of an algebraic group over $F$, and we let $G = \Res_{F/\Q} B^\times$. If $\p_1, \dots, \p_r$ are the primes of $F$ above $p$, then $G$ is isomorphic over $\Q_p$ to $\prod_{i=1}^r G^{(i)}$ where $G^{(i)} = \Res_{F_{\p_i} / \Q_p} B^\times$.

Let us suppose without loss of generality that $B$ is split over $F_{\p_i}$ for $i = 1, \dots, t$ and non-split for $i = t+1, \dots, r$. For $i = 1, \dots, t$, the factor $G^{(i)}$ will be isomorphic to $\Res_{F_{\p_i}/\Q_p} \GL_2$; let us fix such an isomorphism, and define $S^{(i)}$ to be the rank 2 diagonal split torus, and $P_{\rm min}^{(i)}$ the Borel subgroup of upper triangular matrices. For the remaining primes, we take $S^{(i)}$ to be the rank 1 torus corresponding to the inclusion $\Q_p^\times \into B^\times$, and $P_{\rm min}^{(i)}$ the whole of $G^{(i)}$. Then $S = \prod S^{(i)}$ is a maximal split torus in $G$ and $P_{\rm min} = \prod P_{\rm min}^{(i)}$ a minimal parabolic containing $S$; and the parabolic subgroups containing $P_{\rm min}$ biject with the subsets of $\{1, \dots, t\}$.

Let us fix such a parabolic, corresponding (without loss of generality) to $\{1, \dots, s\}$ for some $s \le t$. We can define a compact open subgroup $G_0$ to be the product of the Iwahori subgroups of $G^{(i)}$ for $i = 1, \dots, s$ with the unit groups of maximal orders in the remaining factors. This is admissible, since in each factor the kernel of reduction modulo a sufficiently high power of $\p_i$ will be a good $\Q_p$-analytic open subgroup, which is trivally decomposable (since each factor has semisimple rank 1 or 0). So for any finite dimensional arithmetical locally $\Q_p$-analytic representation $V$ of $M_0$, we have a family of spaces $\Ind(V)_k = C^{\an}(\mathcal{D}_k, V)$, where $\mathcal{D}_k$ is a direct product of terms $\mathcal{D}_k^{(i)}$ which are type 1 thickenings of $\OO_{F, \p_i}$; and hence there is an eigenvariety (for any choice of tame level) interpolating automorphic forms for $D^\times$ which are finite slope for the primes $\p_1, \dots, \p_s$.

The constructions in \cite[\S 1.2]{yamagami} and \cite[part III]{buzzard-eigen}, on the other hand, use the corresponding products of type 2 thickenings of $\OO_{F, \p_i}$. Applying corollary \ref{corr:thickenings}, we obtain natural maps between these representations, and hence between the resulting spaces of overconvergent automorphic forms; part (iii) of the corollary shows that these are links for the action of the compact Hecke operator corresponding to $\begin{pmatrix} 1 & 0 \\ 0 & p^m \end{pmatrix}$ for sufficiently large $m$, and consequently the resulting eigenvarieties are isomorphic.

\section{Acknowledgements}

This paper is a more polished version of the first three chapters of my 2008 University of London PhD thesis, and I am much indebted to my supervisor, Kevin Buzzard, for all the advice and support I received during its preparation. The genesis of this project was in conversations I had with several mathematicians during the Eigenvarieties semester at Harvard University in spring 2006, and it took its final form during a visit to the $p$-adic semester at the Tata Institute of Fundamental Research in autumn 2008; I am grateful to the organisers of these two programmes for the invitations.

Finally, I would also like to thank the numerous mathematicians with whom I have had useful and interesting conversations, particularly Ga\"etan Chenevier, Matthew Emerton, Richard Hill, Payman Kassaei, Jan Kohlhaase, Alex Paulin and Matthias Strauch; the anonymous referee, whose careful reading of the original draft of this paper was immensely helpful; and, last but not least, Sarah Zerbes.

\def\cprime{$'$}

\newpage
\newcommand{\RR}{\mathbf{R}}
%
%
%
%
%
%
%
%
%
%

\appendix

\section*{Corrigendum}

 \newcommand{\frp}{\mathfrak{p}}
 The purpose of this note is to describe and to correct an unfortunate error in the author's earlier paper \cite{loeffler11}. These errors were apparently first noted by Olivier Ta\"ibi in his forthcoming paper \cite{taibi16}; Ta\"ibi's paper explains a corrected version of the argument under a minor extra hypothesis (that the reductive group $G$ be quasi-split). We explain here how to correct the argument in general.
 
 In \S 2.6 of \cite{loeffler11}, the following alterations need to be made. We use here the notations and assumptions in force in the original paper.
 
 \subsection*{A} In definition 2.6.1, the last sentence should be replaced by ``where $\chi_{\mathrm{alg}}$ is the character by which $T$ acts on $V_{\mathrm{alg}}$'', since $V_{\mathrm{alg}}$ is a representation of the algebraic group $M$, but $T$ may be strictly smaller than $Z(M)$ in general.
 
 \subsection*{B} The parenthetical remark following Definition 2.6.1 is only correct under the slightly stronger assumption that $\eta \in \Sigma^{++}$ (rather than the running assumption in force at that point that $\eta \in \Sigma^+$).
 
 \subsection*{C} Proposition 2.6.3 should read as follows:
 
 \medskip
 {\noindent \textsc{Proposition 2.6.3'}. \itshape Suppose $G$ is split over the field $E$. Let $\tilde B \subseteq P_{\mathrm{min}}$ be a Borel subgroup defined over $E$, and $\tilde S$ a maximal torus of $\tilde B$ containing $S$.
 
 Let $\tilde \Delta_f \subseteq X^\bullet(\tilde S)$ be the set of free simple  positive roots of $P$ with respect to $\tilde S$, and for each $\beta \in \tilde\Delta_f$, let $s_\beta$ be the corresponding reflection in the Weyl group. Let $\rho$ be half the sum of the positive roots with respect to $B$.
 
 Let $\lambda$ be the highest weight of $V_{\mathrm{alg}}$ with respect to $\tilde B$. If
 \[ \sigma < \inf_{\beta \in \tilde\Delta_f} f_\eta\left( -\left[ s_\beta(\lambda + \rho) - (\lambda + \rho) \right] \middle|_{T} \right), \tag{\dag}\]
 then $\sigma$ is a small slope for $\eta$.
 }
 
 \begin{proof}
  We begin by noting the following chain of implications, where we abbreviate $\chi_{\mathrm{alg}}$ by $\chi$.
  \begin{gather*}
   \text{``$\sigma$ is a small slope''}\\
   \Leftrightarrow ``\forall\, y\in X^\bullet(T) \text{ such that } f_\eta(y) \le \sigma,\ \chi_{\mathrm{alg}} - y \text{ is strongly classical''}\\
   \Leftrightarrow ``\forall x \in X^\bullet(T) \text{ not strongly classical}, f_\eta(\chi - x) > \sigma\text{.''}
  \end{gather*}
  To identify which weights fail to be strongly classical, we apply the generalised Bernstein--Gelfand--Gelfand resolution of \cite{lepowsky} to the base-extension of $G$ to $E$. This gives an exact sequence of $\mathfrak{g}$-modules
  \[ \bigoplus_{\beta \in \tilde\Delta_f} \Ver_{\mathfrak{p}}(s_\beta(\lambda + \rho) - \rho) \longrightarrow \Ver_\mathfrak{p}(\lambda) \longrightarrow U_{\mathrm{alg}} \longrightarrow 0, \]
  where we have used $\Ver_\mathfrak{p}(\mu)$, for $\mu \in X^\bullet(\tilde S)$, to denote the generalised Verma module of highest weight $\mu$.
  
  Any character of $\mathfrak{t} = \operatorname{Lie} T$ appearing in a generalised Verma module $\Ver_\frp(\mu)$ is of the form $\mu|_T - y$, where $y$ lies in the cone spanned by the restrictions to $T$ of the free positive roots. Since $f_\eta$ takes non-negative values on this cone, we see that the infimum of $f_\eta(\chi - x)$ over all $x$ appearing in the decomposition of $\Ver_\frp(\mu)$ as a $\mathfrak{t}$-representation is attained for $x = \mu|_{T}$. Hence the infimum of $f_\eta(\chi - x)$ over all non-strongly-classical $x$ is attained for some $x$ in the finite set 
  \[ \{ \left[s_\beta(\lambda + \rho) - \rho\right]|_T : \beta \in \tilde \Delta_f\}.\]
  Since we also have $\chi = \lambda \mid_T$, this gives the claim.
 \end{proof}
 
 In place of Proposition 2.6.4 we will use the following somewhat less explicit statement. Recall that $H = M / M^{\mathrm{ss}}$, the quotient of $M$ by its derived subgroup, which is a torus over $L$. There is a natural map from $X^\bullet(H)$ to $X^\bullet(\tilde S)$, given by restriction of characters.
 
 \medskip
  {\noindent \textsc{Proposition 2.6.4'}. \itshape Let us fix $\sigma \in \RR_{> 0}$, and some dominant $\lambda_0 \in X^\bullet(\tilde S)$; and assume $\eta \in \Sigma^{++}$. Then the set of characters $\tau \in X^\bullet(H) \otimes \RR$ such that $\lambda = \lambda_0 + \tau$ is dominant and satisfies (\dag) is a non-empty open cone.
 }
 \begin{proof}
  To each simple root $\beta \in \tilde\Delta_f$ there is associated a coroot $\beta^\vee \in X_{\bullet}(\tilde S)$, and it is a standard fact that the images of these coroots form a basis of $X_\bullet(H) \otimes \RR$. So any set of inequalities of the form $\{ \langle x, \beta_i^\vee\rangle > r_i : 1 \le i \le \#\tilde\Delta_f\}$ defines a non-empty open cone in $X^\bullet(H) \otimes \RR$, and if the $r_i$ are non-negative, any $\tau$ in this cone clearly has the property that $\lambda_0 \otimes \tau$ is dominant.
  
  Since we have
  \[ 
   (\lambda + \rho) - s_\beta(\lambda + \rho) = \langle \lambda+\rho, \beta^\vee\rangle \cdot \beta, 
  \]
  and $f_\eta(\beta|_{T}) > 0$ for all $\beta \in \tilde\Delta_f$ by assumption, the set of $\tau$ such that $\lambda_0 + \tau$ is dominant and satisfies (\dag) has precisely this form.
 \end{proof}
 
 With these corrections, the results of \S 3.13 of the paper now hold as stated, save for the minor modification that we must replace $X^\bullet(T)$ with $X^\bullet(H)$ throughout.


\begin{thebibliography}{10}

\bibitem{ash-stevens07}
Avner Ash and Glenn Stevens.
\newblock {$p$}-adic deformations of arithmetic cohomology.
\newblock Preprint available from \url{http://www2.bc.edu/~ashav/}, October
  2007.

\bibitem{bellaiche-chenevier}
Jo{\"e}l Bella{\"i}che and Ga{\"e}tan Chenevier.
\newblock {$p$}-adic families of {G}alois representations and higher rank
  {S}elmer groups.
\newblock {\em Ast{\'e}risque}, 324, 2009.

\bibitem{bgg-categoryo}
I.~N. Bern\v{s}te{\u\i}n, I.~M. Gel{\cprime}fand, and S.~I. Gel{\cprime}fand.
\newblock A certain category of {$\mathfrak{g}$}-modules.
\newblock {\em Funkcional. Anal. i Prilo\v zen.}, 10(2):1--8, 1976.

\bibitem{borel-tits}
Armand Borel and Jacques Tits.
\newblock Groupes r{\'e}ductifs.
\newblock {\em Pub. Math. IH{\'E}S}, 27:55--151, 1965.

\bibitem{bourbaki-VAR}
Nicolas Bourbaki.
\newblock {\em \'{E}l\'ements de math\'ematique. {F}asc. {XXXIII}.
  {V}ari\'et\'es diff\'erentielles et analytiques. {F}ascicule de r\'esultats
  ({P}aragraphes 1 \`a 7)}.
\newblock Actualit\'es Scientifiques et Industrielles, No. 1333. Hermann,
  Paris, 1967.

\bibitem{buzzard-continuation}
Kevin Buzzard.
\newblock Analytic continuation of overconvergent eigenforms.
\newblock {\em J. Am. Math. Soc.}, 16(1):29--55 (electronic), 2003.

\bibitem{buzzard-families}
Kevin Buzzard.
\newblock On {$p$}-adic families of automorphic forms.
\newblock In {\em Modular curves and abelian varieties (Bellaterra, 2002)},
  volume 224 of {\em Progr. Math.}, pages 23--44. Birkh{\"a}user, 2004.

\bibitem{buzzard-eigen}
Kevin Buzzard.
\newblock Eigenvarieties.
\newblock In {\em {$L$}-functions and {G}alois representations ({D}urham,
  2004)}, volume 320 of {\em London Math. Soc. Lecture Notes}, pages 59--120.
  Cambridge Univ. Press, 2007.

\bibitem{casselman-book}
Bill Casselman.
\newblock Introduction to the theory of admissible representations of
  {$p$}-adic groups.
\newblock Available from
  \url{http://www.math.ubc.ca/~cass/research/p-adic-book.dvi}, 1995.

\bibitem{chenevier-families}
Ga{\"e}tan Chenevier.
\newblock Familles {$p$}-adiques de formes automorphes pour {${\rm GL}\sb n$}.
\newblock {\em J. Reine Angew. Math.}, 570:143--217, 2004.

\bibitem{cogdell-et-al}
James~W. Cogdell, Henry~H. Kim, and M.~Ram Murty.
\newblock {\em Lectures on automorphic {$L$}-functions}, volume~20 of {\em
  Fields Institute Monographs}.
\newblock American Mathematical Society, Providence, RI, 2004.

\bibitem{C-CO}
Robert~F. Coleman.
\newblock Classical and overconvergent modular forms.
\newblock {\em Invent. Math.}, 124(1-3):215--241, 1996.

\bibitem{Cpadic}
Robert~F. Coleman.
\newblock {$p$}-adic {B}anach spaces and families of modular forms.
\newblock {\em Invent. Math.}, 127(3):417--479, 1997.

\bibitem{CMeigen}
Robert~F. Coleman and Barry Mazur.
\newblock The eigencurve.
\newblock In {\em Galois representations in arithmetic algebraic geometry
  (Durham, 1996)}, volume 254 of {\em London Math. Soc. Lecture Notes}, pages
  1--113. Cambridge Univ. Press, 1998.

\bibitem{dixmier}
Jacques Dixmier.
\newblock {\em Enveloping algebras}, volume~11 of {\em Graduate Studies in
  Mathematics}.
\newblock American Mathematical Society, Providence, RI, 1996.
\newblock Revised reprint of the 1977 translation.

\bibitem{emerton-memoir}
Matthew Emerton.
\newblock Locally analytic vectors in representations of locally {$p$}-adic
  analytic groups.
\newblock Mem. Am. Math. Soc. (to appear), 2004.

\bibitem{emerton-jacquet}
Matthew Emerton.
\newblock Jacquet modules of locally analytic representations of {$p$}-adic
  reductive groups. {I}. {C}onstruction and first properties.
\newblock {\em Ann. Sci. {\'E}cole Norm. Sup. (4)}, 39(5):775--839, 2006.

\bibitem{emerton-interpolation}
Matthew Emerton.
\newblock On the interpolation of systems of eigenvalues attached to
  automorphic {H}ecke eigenforms.
\newblock {\em Invent. Math.}, 164(1):1--84, 2006.

\bibitem{emerton-jacquet2}
Matthew Emerton.
\newblock Jacquet modules of locally analytic representations of {$p$}-adic
  reductive groups {II}. {T}he relation to parabolic induction.
\newblock J. Inst. Math. Jussieu (to appear), 2007.

\bibitem{feaux}
Christian~Tobias F{\'e}aux~de Lacroix.
\newblock Einige {R}esultate {\"u}ber die topologischen {D}arstellungen
  {$p$}-adischer {L}iegruppen auf unendlich dimensionalen {V}ektorr{\"a}umen
  {\"u}ber einem {$p$}-adischen {K}{\"o}rper.
\newblock {\em Schr. Math. Inst. Univ. M{\"u}nster (3rd ser.)}, 23:x+111, 1999.

\bibitem{gross-algebraic}
Benedict~H. Gross.
\newblock Algebraic modular forms.
\newblock {\em Israel J. Math.}, 113:61--93, 1999.

\bibitem{hill-banach}
Richard Hill.
\newblock On {E}merton's {$p$}-adic {B}anach spaces.
\newblock {\em Int. Math. Res. Not.}, art. ID rnq042, 2010.

\bibitem{owen-transfer}
Owen Jones.
\newblock A structure theorem for {C}henevier's representations
  {$\mathcal{V}_\lambda$}.
\newblock preprint, Imperial College, February 2008.
\newblock \url{http://www.ma.ic.ac.uk/~ojones/}.

\bibitem{lepowsky}
James Lepowsky.
\newblock A generalization of the {B}ernstein-{G}elfand-{G}elfand resolution.
\newblock {\em J. Algebra}, 49(2):496--511, 1977.

\bibitem{platonov-rapinchuk}
Vladimir Platonov and Andrei Rapinchuk.
\newblock {\em Algebraic groups and number theory}, volume 139 of {\em Pure and
  Applied Mathematics}.
\newblock Academic Press, 1994.

\bibitem{ST-fourier}
Peter Schneider and Jeremy Teitelbaum.
\newblock {$p$}-adic {F}ourier theory.
\newblock {\em Doc. Math.}, 6:447--481 (electronic), 2001.

\bibitem{ST-ugfinite}
Peter Schneider and Jeremy Teitelbaum.
\newblock {$U(\mathfrak{g})$}-finite locally analytic representations.
\newblock {\em Represent. Theory}, 5:111--128, 2001.
\newblock With an appendix by Dipendra Prasad.

\bibitem{ST-distributions}
Peter Schneider and Jeremy Teitelbaum.
\newblock Locally analytic distributions and {$p$}-adic representation theory,
  with applications to {${\rm GL}\sb 2$}.
\newblock {\em J. Am. Math. Soc.}, 15(2):443--468, 2002.

\bibitem{taylor-relative}
Joseph~L. Taylor.
\newblock Homology and cohomology for topological algebras.
\newblock {\em Advances in Math.}, 9:137--182, 1972.

\bibitem{yamagami}
Atsushi Yamagami.
\newblock On {$p$}-adic families of {H}ilbert cusp forms of finite slope.
\newblock {\em J. Number Theory}, 123(2):363--387, 2007.

\end{thebibliography}

\begin{thebibliography}{Loe11}

\bibitem[Lep77]{lepowsky}
J. Lepowsky,
  \emph{A generalisation of the Bernstein--Bernstein--Gelfand resolution}, J. Algebra 49 (1977), 496--511. DOI: \href{http://dx.doi.org/10.1016/0021-8693(77)90254-X}{\texttt{10.1016/0021-8693(77)90254-X}}.
  
\bibitem[Loe11]{loeffler11}
D. Loeffler,
  \emph{Overconvergent
  algebraic automorphic forms}, Proc. London Math. Soc. \textbf{102} (2011),
  no.~2, 193--228. DOI: \href{http://dx.doi.org/10.1112/plms/pdq019}{\texttt{10.1112/plms/pdq019}}

\bibitem[Tai16]{taibi16}
O. Ta\"ibi,
  \emph{Eigenvarieties for classical groups and complex conjugations in Galois representations},
  Math. Res. Lett. \textbf{23} (2016), no.~4, 1167--1220. DOI: \href{http://dx.doi.org/10.4310/MRL.2016.v23.n4.a10}{\texttt{10.4310/MRL.2016.v23.n4.a10}}.

\end{thebibliography}
\end{document}